\documentclass[12pt]{amsart}

\usepackage{amscd,latexsym,amsthm,amsfonts,amssymb,amsmath,amsxtra}
\usepackage[mathscr]{eucal}
\usepackage{hyperref}
\renewcommand\theequation{\thesection.\arabic{equation}}

\newcommand{\BC}{{\mathbb {C}}}

\newcommand{\BQ}{{\mathbb {Q}}}
\newcommand{\BR}{{\mathbb {R}}}

\newcommand{\BZ}{{\mathbb {Z}}}

\makeatletter
\def\Ddots{\mathinner{\mkern1mu\raise\p@
\vbox{\kern7\p@\hbox{.}}\mkern2mu
\raise4\p@\hbox{.}\mkern2mu\raise7\p@\hbox{.}\mkern1mu}}
\makeatother

\newcommand{\CA}{{\mathcal {A}}}
\newcommand{\CB}{{\mathcal {B}}}

\newcommand{\CF}{{\mathcal {F}}}

\newcommand{\CH}{{\mathcal {H}}}

\newcommand{\RG}{{\mathrm {G}}}

\newcommand{\RT}{{\mathrm {T}}}

\newcommand{\RV}{{\mathrm {V}}}

\newcommand{\RZ}{{\mathrm {Z}}}

\newcommand{\GL}{{\mathrm{GL}}}

\newcommand{\Hom}{{\mathrm{Hom}}}

\newcommand{\Ind}{{\mathrm{Ind}}}

\newcommand{\SL}{{\mathrm{SL}}}

\newcommand{\wt}{\widetilde}

\newcommand{\ol}{\overline}

\def\diag{{\rm diag}}

\newtheorem{thm}{Theorem}[section]
\newtheorem{cor}[thm]{Corollary}
\newtheorem{lem}[thm]{Lemma}
\newtheorem{prop}[thm]{Proposition}
\newtheorem {conj}[thm]{Conjecture}

\newtheorem {ques/conj}[thm]{Question/Conjecture}

\newtheorem{defn}[thm]{Definition}
\newtheorem{rmk}[thm]{Remark}

\makeatletter

\newcommand{\Rmnum}[1]{\expandafter\@slowromancap\romannumeral #1@}
\makeatother

\begin{document}
\renewcommand{\theequation}{\arabic{equation}}
\numberwithin{equation}{section}

\title[LLC for Simple Supercuspidals of $\GL_n(F)$]{The Local Langlands Correspondence for Simple Supercuspidal Representations of $\GL_n(F)$}

\author{Moshe Adrian}
\address{Department of Mathematics\\
University of Utah\\
Salt Lake City, UT 84112, U.S.A.}
\email{madrian@math.utah.edu}

\author{Baiying Liu}
\address{Department of Mathematics\\
University of Utah\\
Salt Lake City, UT 84112, U.S.A.}
\email{liu@math.utah.edu}

\begin{abstract}
Let $F$ be a non-archimedean local field of characteristic zero with residual characteristic $p$.
In this paper, we present a simple proof and construction of the local Langlands correspondence for simple supercuspidal representations of $\GL_n(F)$, when $p \nmid n$.  Our proof relies on the existence of the local Langlands correspondence for $\GL_n(F)$, due to Harris/Taylor and Henniart.  As an application, we prove Jacquet's conjecture on the local converse problem for $\GL_n(F)$ in the case of simple supercuspidal representations.
\end{abstract}

\maketitle

\tableofcontents

\section{Introduction}
The local Langlands conjecture for $\GL_n(F)$ states that there exists a certain bijection between representations of the Weil-Deligne group and representations of $\GL_n(F)$, where $F$ is a $p$-adic field.  Although the conjecture has been proven by Harris/Taylor and Henniart (see \cite{HT01}, \cite{H00}), the correspondence has not completely been made explicit.  In a series of papers \cite{BH05i, BH05ii, BH10}, Bushnell and Henniart have explicitly described the correspondence in great generality.

Recently, Gross, Reeder, and Yu \cite{GR10, RY13}, have studied a new class of supercuspidal representations of $p$-adic groups, called \emph{epipelagic} supercuspidal representations.  Kaletha has recently provided an explicit construction of a local Langlands correspondence for these representations, for general tamely ramified reductive $p$-adic groups \cite{K13}.  In the case of $\GL_n(F)$ and $p \nmid n$, the local Langlands correspondence for these supercuspidal representations is already covered in \cite{BH05ii}.  What remained to be constructed in the setting of epipelagic supercuspidal representations of general linear groups was the correspondence in the case of $\GL_{p^r}(F)$.  Bushnell and Henniart have recently carried out this construction in \cite{BH13}, for a certain class of epipelagic supercuspidal representations, called the \emph{simple} supercuspidal repesentations.

It turns out that in the case of $p \nmid n$, it is possible to give a simple proof and construction of the local Langlands correspondence for simple supercuspidal representations of $\GL_n(F)$, without using the technical machinery of Bushnell, Henniart, and Kutzko.  In this paper, we present this proof and construction.  Because the depth zero correspondence rests on the rich theory of cuspidal representations of general linear groups over finite fields, it seems to us that the $p \nmid n$ simple supercuspidal case might be the simplest and least demanding of all cases of the local Langlands correspondence.

Our paper is split into three parts.  In the first part, we present our simple proof and construction of the local Langlands correspondence for simple supercuspidal representations of $\GL_n(F)$ when $p \nmid n$.  To carry this out, we explicitly compute a family of twisted epsilon factors on both sides of the local Langlands correspondence in this setting.  Traditionally, epsilon factors are not computed explicitly.  However, in the case of simple supercuspidal representations, they are tractable.  Given a simple supercuspidal representation $\pi$ of $\GL_n(F)$, we define a function in the Whittaker model of $\pi$ with which we can compute zeta integrals without much difficulty.  This leads, via the theory of the local functional equation, to explicit values for a family of twisted epsilon factors of $\pi$.  We then attach to $\pi$ a conjectural Langlands parameter, and compute a family of its twisted epsilon factors as well.  These involve Gauss sums, but in our setting, they can also be computed explicitly.  After computing these epsilon factors, we describe and prove the local Langlands correspondence (see Theorem \ref{maintheorem1}), using its existence which was proven by Harris/Taylor and Henniart.

The second part of the paper concerns epipelagic supercuspidal representations.  In \cite{RY13}, Reeder and Yu have constructed a class of epipelagic supercuspidal representations for semisimple $p$-adic groups $G$, which naturally generalizes the class of simple supercuspidal representations. A natural question to ask is whether or not their construction exhausts all epipelagic supercuspidal representations of $G$.  In \S\ref{epipelagicissimple}, we first show that their construction generalizes to $\GL_n(F)$.  We then show that our constructed epipelagic supercuspidals of $\GL_n(F)$ can only come from a barycenter of an alcove in the building of $\GL_n(F)$. Since there are epipelagic supercuspidal representations of $\GL_n(F)$ that do not come from a barycenter of an alcove (see Remark \ref{epipelagicisnotsimple}), we conclude that our constructed epipelagic supercuspidals of $\GL_n(F)$ do not necessarily exhaust all epipelagic supercuspidals of $\GL_n(F)$.  This trivially implies (by considering the case of $\SL_n(F)$) that the construction of Reeder and Yu is also not necessarily exhaustive. Therefore, the construction of Reeder and Yu should be thought of as an initial step towards the construction of all epipelagic supercuspidal representations for general $p$-adic groups.

We would like to remark that in \cite{BH13}, Bushnell and Henniart only consider the set of simple supercuspidal representations of $\GL_n(F)$, which is in many cases a proper subset of the set of epipelagic supercuspidal representations of $\GL_n(F)$ (see Remark \ref{epipelagicisnotsimple}).  Indeed, the supercuspidals that they consider have swan conductor $1$, and it can be seen that these supercuspidals are exactly the simple ones.

Finally, in the third part of the paper, we give an application of our work to Jacquet's conjecture on the local converse problem for $\GL_n(F)$, in the case of simple supercuspidal representations.  Explicitly, we show that any two unitarizable simple supercuspidal representations of $\GL_n(F)$ with the same central character have a special pair of Whittaker functions (see Theorem \ref{thmepi}), building on recent work of Jiang, Nien, and Stevens \cite{JNS13}.  This proves Jacquet's conjecture in a previously unknown case.  Our result appears generalizable, and this is currently work in progress.

We now present an outline of the paper.  In \S\ref{llcsimple}, we first recall the notion of a simple supercuspidal representation of $\GL_n(F)$ and some basic notions on epsilon factors.  We then compute a family of twisted epsilon factors of simple supercuspidal representations of $\GL_n(F)$, as well as a family of twisted epsilon factors of certain Langlands parameters.  In the process, we prove and construct the local Langlands correspondence for simple supercuspidal representations of $\GL_n(F)$.  In \S\ref{epipelagicissimple}, we begin by giving a construction of a class of epipelagic supercuspidal representations of $\GL_n(F)$.  We then prove two results, having immediate consequences for the epipelagic supercuspidals of $\GL_n(F)$ that we construct.  The first is that the only barycenters in the building of $\GL_n(F)$ that admit stable functionals are barycenters of alcoves (see Theorem \ref{main2}).  We then prove that a nonbarycenter point in the building cannot admit a certain class of semi-stable functionals (see Theorem \ref{main3}).  These two results imply that our constructed epipelagic supercuspidal representations are simple.  In \S\ref{jacquetsection}, we begin by recalling Jacquet's conjecture on the local converse problem for $\GL_n(F)$, and a new strategy, due to Jiang, Nien, and Stevens, to approach Jacquet's conjecture.  We then prove that any two unitarizable simple supercuspidal representations of $\GL_n(F)$ with the same central character admit a special pair of Whittaker functions, thereby proving Jacquet's conjecture in a new case.

\subsection*{Acknowledgements}

This paper has benefited from conversations with Stephen DeBacker, Dihua Jiang, Mark Reeder, Gordan Savin, Freydoon Shahidi, Shaun Stevens, Shuichiro Takeda, and Geo Kam-Fai Tam.  We thank them all.

\section{Notation}\label{notation}
Let $F$ be a non-archimedean local field of characteristic zero.  We let $\mathfrak{o}$ denote its ring of integers, $\mathfrak{p}$ the maximal ideal in $\mathfrak{o}$, and $k_F$ the residue field.  Fix a uniformizer $\varpi$ in $F$, and let $val$ denote valuation on $F$.  If $E/F$ is a finite extension, we use $\mathfrak{o}_E$ and $\mathfrak{p}_E$ to denote its associated ring of integers and maximal ideal.  $N_{E/F}$ will denote the norm map from $E$ to $F$.  The \emph{level} of a character $\psi \in \widehat{F}$ will be the smallest integer $c$ such that $\psi|_{\mathfrak{p}^c} \equiv 1$.  The \emph{level} of a character $\chi \in \widehat{F^{\times}}$ will be the smallest nonnegative integer $d$ such that $\chi|_{1 + \mathfrak{p}^{d+1}} \equiv 1$.  A character $\chi \in \widehat{F^{\times}}$ is called \emph{tamely ramified} if its level is zero.  We fix a nontrivial additive character $\psi$ of $F$ of level one.

Let $G = \GL_n(F)$, $Z$ the center of $G$, and $\GL_n(\mathfrak{o})$ the standard maximal compact subgroup.  We denote by $I$ the standard Iwahori subgroup consisting of matrices which modulo $\mathfrak{p}$ are the standard upper triangular Borel of $\GL_n(k_F)$, we let $I^+$ be its pro-unipotent radical. Let $\mathrm{T}$ the diagonal maximal torus of $\GL_n$ and set $T = \mathrm{T}(F)$.  Let $N$ denote the normalizer of $T$ in $G$, and let $T_1 = \mathrm{T}(1 + \mathfrak{p})$.  Set $W = N / T$ and $W_1 = N / T_1$, and let $U_n$ denote the standard maximal unipotent subgroup of $G$.  We will also sometimes write $U$ for $U_n$, when $n$ is clear.  For any $u \in U_n$, let
$$\psi_{U_n}(u)=\psi(\sum_{i=1}^{n-1} u_{i,i+1}),$$ the standard non-degenerate character of $U_n$.  Finally, let $\mathrm{M}_{r \times s}(F)$ denote the space of $r \times s$ matrices with coefficients in $F$.  We will sometimes write $\mathrm{M}_{r \times s}$ when the field is clear.

We will use ``$\mathrm{ind}$" to denote compact induction.
Moreover, we will fix throughout a self-dual Haar measure on $F$, relative to $\psi$.  We have in particular that $\int_{\mathfrak{o}} dx = q^{1/2}$.  It will be convenient to fix a Haar measure $d^* x$ on $F^{\times}$ such that $\int_{\mathfrak{o}^{\times}} d^* x = 1$.

\section{The local Langlands correspondence for simple supercuspidal representations of $\GL_n(F)$}\label{llcsimple}

In this section, we will give a proof and construction of the local Langlands correspondence for simple supercuspidal representations of $\GL_n(F)$, when $p \nmid n$.  In \S\ref{prelimsimplesupercuspidal}, we recall the definition of simple supercuspidal representation of $\GL_n(F)$. The basic theory of epsilon factors of pairs for $\GL_n(F)$ is recalled in \S\ref{preliminariesepsilonfactors}.  In \S\ref{computingepsilonfactors}, we compute the standard epsilon factor of a simple supercuspidal representation of $\GL_n(F)$, and in \S\ref{langlandsgln}, we make a prediction for the Langlands parameter associated to a simple supercuspidal representation of $\GL_n(F)$.  We compute the standard epsilon factor of our predicted Langlands parameter in \S\ref{computationepsilon}.  Finally, in \S\ref{epsilontwists}, we compute the twisted epsilon factors of a simple supercuspidal representation and its predicted Langlands parameter, where we twist by tamely ramified characters of $\GL_1(F)$.  These computations, along with some well-known results, are enough for us to construct and prove the local Langlands correspondence for simple supercuspidal representations of $\GL_n(F)$, when $p \nmid n$ (see Theorem \ref{maintheorem1}).

\subsection{Preliminaries on simple supercuspidal representations of $\GL_n(F)$}\label{prelimsimplesupercuspidal}

We begin by reviewing the definition of simple supercuspidal representation of $\GL_n(F)$, as in \cite{KL13}.
We set $H = ZI^+$.  Fix a character $\omega$ of $Z$, trivial on $1 + \mathfrak{p}$.  For $(t_1, t_2, ..., t_n) \in \mathfrak{o}^{\times} / (1 + \mathfrak{p}) \times \mathfrak{o}^{\times} / (1 + \mathfrak{p}) \times \cdots \times \mathfrak{o}^{\times} / (1 + \mathfrak{p})$, we define a character $\chi : H \rightarrow \mathbb{C}^{\times}$ by $\chi(zk) = \omega(z) \psi(t_1 r_1 + ... + t_n r_n)$ for $z \in Z$ and

\[
k = \left( \begin{array}{ccccc}
x_1 & r_1 & * &  \cdots & \\
* & x_2 & r_2 & \cdots & \\
\vdots & & \ddots & \ddots & \\
* & &  &  & r_{n-1}\\
\varpi r_n &  & \cdots &  & x_n
\end{array} \right) \in I^+.
\]

These $\chi$'s are called the \emph{affine generic characters} of $H$.  By \cite[Theorem 3.4]{KL13}, the orbits of affine generic characters are parameterized by the set of elements in $\mathfrak{o}^{\times} / (1 + \mathfrak{p})$, as follows.  $T \cap \GL_n(\mathfrak{o})$ normalizes $H$, so acts on the set of affine generic characters.  Every orbit of affine generic characters contains one of the form $(1,1,...,1,t)$, for $t \in \mathfrak{o}^{\times} / (1 + \mathfrak{p})$.  Specifically, the orbit of $(t_1, t_2, ..., t_n)$ contains $(1,1, ..., 1, t)$, where $t = t_1 t_2 \cdots t_n$.

Instead of viewing the affine generic characters as parameterized by $t \in \mathfrak{o}^{\times} / (1 + \mathfrak{p})$, we will set $t = 1$ and let the affine generic characters be parameterized by the various choices of uniformizer in $F$.  Since we have already fixed an (arbitrary) uniformizer ahead of time in \S\ref{notation}, we have therefore fixed an affine generic character $\chi$.  The compactly induced representation $\pi_{\chi} := \mathrm{ind}_H^G \chi$ is a direct sum of $n$ distinct irreducible supercuspidal representations of $\GL_n(F)$. They are parameterized by $\zeta$, where $\zeta$ is a complex $n^{\mathrm{th}}$ root of $\omega(\varpi)$, as follows.  Set

\[
g_{\chi} = \left( \begin{array}{ccccc}
0 & 1 &  &   & \\
 &  & 1 &  & \\
 & & & \ddots & \\
 & &  &  & 1\\
\varpi  &  &  &  & 0
\end{array} \right)
\]

Set $H' = \langle g_{\chi} \rangle H$.  Then the summands of $\pi_{\chi}$ are the compactly induced representations
$$
\sigma_{\chi}^{\zeta} := \mathrm{ind}_{H'}^G \chi_{\zeta}
$$
where $\chi_{\zeta}(g_{\chi}^j h) = \zeta^j \chi(h)$, as $\zeta$ runs over the complex $n^{\mathrm{th}}$ roots of $\omega( \varpi)$.  The $\sigma_{\chi}^{\zeta}$'s are the \emph{simple supercuspidal} representations of $G$.

\subsection{Preliminaries on $\epsilon$-factors for $\GL_n(F)$}\label{preliminariesepsilonfactors}

In this section, we recall the local functional equation.  A reference for these results is \cite{Cog00}.

We set $w_{n,m} =
\left( \begin{array}{cc}
I_m & 0\\
0 & w_{n-m}
\end{array} \right) \in \GL_n(F)$ with $w_r =
\left( \begin{array}{ccc}
 & & 1\\
 & \Ddots &\\
 1 & &
\end{array} \right) \in \GL_r(F)$, where $I_m$ denotes the $m \times m$ identity matrix.  Suppose that $\pi$ is a generic representation of $\GL_n(F)$ and $\pi'$ is a generic representation of $\GL_m(F)$.  Let $\mathcal{W}(\pi, \psi)$, $\mathcal{W}(\pi', \psi^{-1})$ denote their Whittaker models with respect to $\psi, \psi^{-1}$, respectively.  Note that $\GL_n(F)$ acts on $\mathcal{W}(\pi, \psi)$ by $(\rho(g) W)(h) := W(hg)$ for $h,g \in \GL_n(F)$.  Let $W \in \mathcal{W}(\pi, \psi)$, $W' \in \mathcal{W}(\pi', \psi^{-1})$.   For $W \in \mathcal{W}(\pi, \psi)$, we set $\widetilde{W}(g) = W(w_n {}^t g^{-1})$.  We set

\[
\widetilde{\Psi}(s;W, W') = \int \int W \left( \begin{array}{ccc}
h & &\\
x & I_{n-m-1} &\\
  & & 1
\end{array} \right) dx \ W'(h) |det(h)|^{s-(n-m)/2} dh
\]
where $h$ is integrated over $U_m(F) \backslash \GL_m(F)$, and where $x$ is integrated over $\mathrm{M}_{(n-m-1) \times m}(F)$.  We also set

\[
\Psi(s;W, W') = \int_{U_m(F) \setminus \GL_m(F)} W \left( \begin{array}{cc}
h & 0\\
0 & I_{n-m}
\end{array} \right) W'(h) |det(h)|^{s-(n-m)/2} dh
\]
Let $\omega'$ denote the central character of $\pi'$.

\begin{thm}\label{gammafactor}
There is a rational function $\gamma(s, \pi \times \pi', \psi) \in \mathbb{C}(q^{-s})$ such that
$$\widetilde{\Psi}(1-s;\rho(w_{n,m}) \widetilde{W}, \widetilde{W}') = \omega'(-1)^{n-1} \gamma(s, \pi \times \pi', \psi) \Psi(s; W, W'), \ if \ m < n$$
for all $W \in \mathcal{W}(\pi, \psi), W' \in \mathcal{W}(\pi', \psi^{-1})$.
\end{thm}

\begin{thm}\label{lfunctionstrivial}
If $\pi$ and $\pi'$ are both (unitary) supercuspidal and if $m < n$, then $L(s, \pi \times \pi') \equiv 1$.
\end{thm}

\begin{defn}\label{epsilonfactordefinition}
The local factor $\epsilon(s, \pi \times \pi', \psi)$ is defined as the ratio
$$\epsilon(s, \pi \times \pi', \psi) = \frac{\gamma(s, \pi \times \pi', \psi) L(s, \pi \times \pi')}{L(1-s, \widetilde{\pi} \times \widetilde{\pi}')}$$ where $\widetilde{\pi}, \widetilde{\pi}'$ denote the contragredients of $\pi, \pi'$, respectively.
\end{defn}

\subsection{Standard $\epsilon$-factors for simple supercuspidals of $\GL_n(F)$}\label{computingepsilonfactors}

In this section, we compute the standard $\epsilon$-factors for the simple supercuspidal representations of $\GL_n(F)$.  Namely, in the notation of \S\ref{preliminariesepsilonfactors}, we set $m = 1$ and $\pi'$ to be the trivial character of $\GL_1(F)$.  We set $\pi_{\zeta} = \sigma_{\chi}^{\zeta}$, for $\zeta$ a complex $n^{\mathrm{th}}$ roots of $\omega( \varpi)$, as in \S\ref{prelimsimplesupercuspidal}.

We begin by defining a Whittaker function on $\GL_n(F)$ by setting

\begin{equation*}
W(g) = \left\{
\begin{array}{rll}
\psi_U(u) \chi_{\zeta}(h') & \text{if} & g = uh' \in U H'\\
0 &  & \text{else}
\end{array} \right.
\end{equation*}
Note that this function is well-defined, by definition of $\psi_U$ and $\chi_{\zeta}$, and it is in $\mathcal{W}(\pi, \psi)$.  This Whittaker function is a special case of those written down by Bushnell/Henniart (see \cite{BH98}) and Paskunas/Stevens (see \cite{PS08}).  An elementary computation shows that
\begin{equation}\label{sec3equ1}
\widetilde{\Psi}(1-s;\rho(w_{n,m}) \widetilde{W}, \widetilde{W}') = \int \int W(Y) dx_1 dx_2 \cdots dx_{n-2} |h|^{1-s-(n-1)/2} dh
\end{equation}
where \[Y = \left( \begin{array}{cccccc}
0 & 1 & 0 & 0 & \cdots & 0\\
0 & 0 & 1 & 0 & \cdots & 0\\
0 & 0 & 0 & 1 & \cdots & 0\\
\vdots & \vdots & \vdots & \vdots & \ddots & \vdots\\
0 & 0 & 0 & 0 & 0 & 1\\
h^{-1} & 0 & -\frac{x_{n-2}}{h} & -\frac{x_{n-3}}{h} & \cdots & -\frac{x_1}{h}
\end{array} \right), \ x_i \in F \ \mathrm{and} \ h \in F^{\times}.
\]

To evaluate $\widetilde{\Psi}(1-s;\rho(w_{n,m}) \widetilde{W}, \widetilde{W}')$, we must determine when $Y$ is in the support of $W$.  We first recall a version of the affine Bruhat decomposition of $G$.

\begin{lem}\label{affinebruhatprelim}
$G = UN I^+$.
\end{lem}

\proof
Set

\[a = \left( \begin{array}{ccccc}
\varpi^{n-1} & 0 & 0 &  \cdots & 0\\
0 & \varpi^{n-2} &  0 & \cdots & 0\\
0 & 0 & \varpi^{n-3} &  \cdots & 0\\
\vdots & \vdots & \vdots &  \ddots & \vdots\\
0 & 0 & 0 & 0 & 1
\end{array} \right). \]
Multiplying $G = I^+ N I^+$ by $a^{-r}$ on both sides we get
$$G = a^{-r} I^+ N I^+ = a^{-r} I^+ a^{r} a^{-r} N I^+ = a^{-r} I^+ a^{r} N I^+,$$  since $a^{-r} N = N$.  Taking the limit as $r \rightarrow \infty$, the claim follows.
\qed

\begin{lem}\label{affinebruhat}
We have a partition
$$G = \displaystyle\coprod_{x \in W_1} U x I^+,$$ where the union is meant to be taken over any set of representatives $x$ in $W_1$.
\end{lem}

\proof
Suppose $u_1 n_1 k_1 = u_2 n_2 k_2$, for $u_i \in U, k_i \in I^+$, and $n_i$ representatives of elements $w_1, w_2 \in W_1$.  The set consisting of all valuations of the entries in $u_i$ is bounded below by some integer, say, $m$.  In particular, $u_i \in a^{r} I^+ a^{-r}$ for some integer $r$.  We have
\[
a^{-r} u_1 a^{r} a^{-r} n_1 k_1 = a^{-r} u_2 a^{r} a^{-r} n_2 k_2.
\]
It is well known that we have a partition $\displaystyle\coprod_{x \in W_1} I^+ x I^+$, so $a^{-r} n_1$ and $a^{-r} n_2$ represent the same class in $W_1$, giving disjointness of the double cosets $U x I^+$.  That these double cosets fill up all of $G$ follows from Lemma \ref{affinebruhatprelim}.
\qed

We now determine which double cosets contain $Y$.  More generally, for $n \geq 2$, we set

\begin{align*}
& A_n = \left\{ \left( \begin{array}{cccccc}
0 & 1 & 0 & 0 & \cdots & 0\\
0 & 0 & 1 & 0 & \cdots & 0\\
0 & 0 & 0 & 1 & \cdots & 0\\
\vdots & \vdots & \vdots & \vdots & \ddots & \vdots\\
0 & 0 & 0 & 0 & 0 & 1\\
a & 0 & b_{n-2} & b_{n-3} & \cdots & b_1
\end{array} \right) : a \in F^{\times}, b_i \in F \right\} ,\\
& g = \left( \begin{array}{cccccc}
0 & g_1 & 0 & 0 & \cdots & 0\\
0 & 0 & g_2 & 0 & \cdots & 0\\
0 & 0 & 0 & g_3 & \cdots & 0\\
\vdots & \vdots & \vdots & \vdots & \ddots & \vdots\\
0 & 0 & 0 & 0 & 0 & g_{n-1}\\
g_n & 0 & 0 & 0 & \cdots & 0
\end{array} \right), \ g_j \in F^{\times}.
\end{align*}

We first determine when elements in $A_n$ are contained in double cosets of the form $U g^j I^+$, for $j \geq 0$.

\begin{lem}\label{support}
Let $X \in A_n$.  If $X \in U g^j Z I^+$ for some $1 \leq j \leq n$, then $j = 1$.
\end{lem}

\proof
The proof is by strong induction on $n$.  The case $n = 2$ is clear, since in this case $X$ is already in the form $g$.  In the rest of the proof, we will use the fixed notation $U, g, Z, I^+$ in the context of general linear groups of possibly different rank.  The context should be clear.

Suppose the claim is true for $n = 2,3,\dots, k$.  Let $X \in A_{k+1}$. Assume $X \in U g^j Z I^+$ for some $1 \leq j \leq k + 1$.
We will show that this implies that $val(a) \leq val(b_i) \ \forall i$.

Suppose first that $val(b_1) < min \{a, val(b_i) : i \geq 2 \}$.  We may then perform leftwards column reduction on $X$ by multiplication on the right by $I^+$, in order to clear $\{ a, b_2, b_3, \cdots b_{n-2} \}$ from the bottom row.  We then perform upwards row reduction by multiplication on the left by $U$, in order to clear the $1$ above $b_1$.  After both of these reductions, we see that the upper left $k \times k$ submatrix $X_k$ of the resulting matrix is contained in $A_k$.

We now use the induction hypothesis.  Suppose first that $X_k \in U g^j Z I^+ \subset \GL_k(F)$ for some $1 \leq j \leq k$.  Then by the induction assumption, via multiplication on the left by $U$ and on the right by $I^+$, we find that $X$ can be reduced to an element of the form

\[Y=\begin{pmatrix}
\wt{Y} & 0 \\
0 & a_n
\end{pmatrix}=
\left( \begin{array}{cccccc}
0 & a_2 & 0 & 0 & \cdots & 0\\
0 & 0 & a_3 & 0 & \cdots & 0\\
\vdots & \vdots & \vdots & \ddots & \vdots & \vdots\\
0 & 0 & 0 & \cdots & a_{n-1} & 0\\
a_1 & 0 & 0 & 0 & 0 & 0\\
0 & 0 & 0 & 0 & \cdots & a_n
\end{array} \right) \in N.\]
This shows easily that $X \notin U g^i Z I^+$ for any $1 \leq i \leq k+1$.
Indeed, by direct calculation, there is {\it a useful observation} that for any
$h \in U g^i Z I^+ \subset \GL_{k+1}(F)$, with $1 \leq i \leq k+1$, we have that
$h_{n, i} \neq 0$.
Therefore, if $X \in U g^i Z I^+$ for some $1 \leq i \leq k+1$, then $i$ must be $k+1$, which implies that $\wt{Y}$ must be in $U g^k Z I^+ \subset \GL_{k}(F)$. On the other hand, since $a_1 \neq 0$, $\wt{Y}$ must be in $U g Z I^+ \subset \GL_{k}(F)$. By Lemma \ref{affinebruhat}, we have a contradiction.

Now suppose that $X_k \notin U g^j Z I^+$ for any $1 \leq j \leq k$. Since $b_1 \neq 0$, if $X$ were to be in $U g^i Z I^+$ for some $1 \leq i \leq k+1$, then by the above observation, $X$ must be in $U g^{k+1} Z I^+ \subset \GL_{k+1}(F)$, which implies that $X_k \in U g^k Z I^+ \subset \GL_{k}(F)$, a contradiction.

Assume now that $val(b_1) \geq min \{a, val(b_i) : i \geq 2 \}$.  Suppose that $val(a) > val(b_i)$ for some $i$. Let $\ell = max \{j : val(b_j) \leq val(b_i) \ \forall i \neq j \}$.  In particular, we have $val(b_{\ell}) < val(a)$.  Multiplication on the right by $I^+$ allows us to perform column reduction, using the column containing $b_{\ell}$, to clear out $\{a, b_i : i \neq \ell \}$ from the bottom row.  This potentially introduces non-zero entries $y_i, i \neq \ell $ into the columns containing the $b_i, where i \neq \ell$, and potentially a nonzero entry $y$ into the column of $a$.  Now we perform row reduction upwards, using multiplication on the left by $U$, to clear the $1$ in the column containing $b_{\ell}$.  Next, we may use row reduction upwards to clear out $\{ y_i : 1 \leq i < \ell \}$, using the ones in those columns. Denote the resulting matrix by $Y$. One can now see that the upper left  $(n - \ell) \times (n - \ell)$ submatrix $\wt{Y}$ of $Y$ is contained in $A_{n - \ell}$.
Using the strong induction hypothesis, Lemma \ref{affinebruhat}, and a similar type of argument as in the previously considered case $val(b_1) < min \{a, val(b_i) : i \geq 2 \}$, we get that $X \notin U g^i Z I^+$ for any $1 \leq i \leq k+1$, a contradiction.

Therefore, we have proven that if $X \in U g^j Z I^+$ for some $1 \leq j \leq k+1$, then $val(a) \leq val(b_i) \ \forall i$.  We may then perform rightwards column reduction on $X$, using multiplication on the right by $I^+$, to clear out all of the $b_i$.  Then $X$ is of the desired form, and together with Lemma \ref{affinebruhat}, the claim is proven.
\qed

We now wish to determine when $Y$ is in the support of $W$.  By definition of $W$, we must write $Y$ as a product $uh'$, where $u \in U, h' \in H'$.  Let us write $UH' = U \langle g_{\chi} \rangle Z I^+$.  By Lemma \ref{support}, we must write $Y$ as a product in $U g_{\chi} Z I^+$.  Part of the proof of Lemma \ref{support} implies that we must have $val(h^{-1}) \leq val(-\frac{x_i}{h}) \ \forall i$.  Therefore, $x_i \in \mathfrak{o} \ \forall i$.

Since $val(h^{-1}) \leq val(-\frac{x_i}{h}) \ \forall i$, then as in the proof of Lemma \ref{support}, multiplication by $I^+$ on the right reduces $Y$ to the matrix

\[Y' = \left( \begin{array}{cccccc}
0 & 1 & 0 & 0 & \cdots & 0\\
0 & 0 & 1 & 0 & \cdots & 0\\
0 & 0 & 0 & 1 & \cdots & 0\\
\vdots & \vdots & \vdots & \vdots & \ddots & \vdots\\
0 & 0 & 0 & 0 & 0 & 1\\
h^{-1} & 0 & 0 & 0 & \cdots & 0
\end{array} \right).
\]

Let $u \in U, z \in Z, k \in I^+$.  Let $z = \diag(a,a, \cdots, a)$, with $a \in F^{\times}$, and $k = (k_{i,j})$.
Assuming that $Y' = u g_{\chi} z k$, it is easy to see that

\[
Y' = u g_{\chi}z k = \left( \begin{array}{ccccc}
 0& a k_{2,2} &  &    & \\
 & 0 & a k_{3,3} & 0 & \\
 &  & \ddots & \ddots & \\
 & 0 &  & 0 & a k_{n,n}\\
\varpi a k_{1,1}  &  &  &  & 0
\end{array} \right).
\]

Hence $a k_{j,j}=1$, for $2 \leq j \leq n$.
Since $k_{j,j} \in 1 + \mathfrak{p} \ \forall j$, we get that $a \in 1 + \mathfrak{p}$.  In particular, $\varpi a k_{1,1} \in \varpi (1 + \mathfrak{p})$.  This says that $h \in \frac{1}{\varpi}(1 + \mathfrak{p}) = \frac{1}{\varpi} + \mathfrak{o}$.  We have therefore concluded that $x_i \in \mathfrak{o} \ \forall i$, and $h \in \frac{1}{\varpi} + \mathfrak{o}$.

\begin{thm}\label{zetaintegral}
$\widetilde{\Psi}(1-s;\rho(w_{n,m}) \widetilde{W}, \widetilde{W}') = \chi_{\zeta}(g_{\chi}) q^{-s-1/2}$.
\end{thm}

\proof
During the column reduction taking us from $Y$ to $Y'$ as in Lemma \ref{support}, we multiply $Y$ on the right by unipotent matrices in $I^+$ whose only nonzero entries are along the diagonal and the
$$(1,3), (1,4), ..., (1,n)$$ entries.  But these types of unipotent matrices are in the kernel of $\chi_{\zeta}$.  Thus, $W(Y) = W(Y')$.  Let us write $Y' = g_{\chi}  k$, where $k$ is a diagonal matrix in $I^+$.  Then $W(Y') = \chi_{\zeta}(g_{\chi} k) = \chi_{\zeta}(g_{\chi})$ by definition of affine generic character.  Therefore,
\begin{align*}
& \widetilde{\Psi}(1-s;\rho(w_{n,m}) \widetilde{W}, \widetilde{W}')\\
= \ & \chi_{\zeta}(g_{\chi}) \int_{\frac{1}{\varpi} + \mathfrak{o}} \int_{\mathfrak{o}} \int_{\mathfrak{o}} ... \int_{\mathfrak{o}}  dx_1 dx_2 \cdots dx_{n-2} |h|^{(1-s)-(n-1)/2} dh \\
= \ & \chi_{\zeta}(g_{\chi})  q^{-s-1/2},
\end{align*}
as we have normalized our measures to satisfy $\int_{\mathfrak{o}} dx_i = q^{1/2}$ and $\int_{\mathfrak{o}^{\times}} dh = 1$.
\qed

\begin{thm}\label{otherzetaintegral}
\[
\Psi(s;W, W') = q^{-1}
\]
\end{thm}

\proof
By definition of $W$,
\[
\Psi(s;W, W') = \int_{F^{\times}} W \left( \begin{array}{cc}
h & 0\\
0 & I_{n-1}
\end{array} \right) |h|^{s-(n-1)/2} dh
\]

\[
= \int_{1 + \mathfrak{p}} W \left( \begin{array}{cc}
h & 0\\
0 & I_{n-1}
\end{array} \right) |h|^{s-(n-1)/2} dh = vol(1 + \mathfrak{p}).
\]
\qed

By Definition \ref{epsilonfactordefinition} and Theorem \ref{lfunctionstrivial}, we now have

\begin{cor}\label{epsilonfactorgln}
$\epsilon(s, \pi_{\zeta}, \psi) = \chi_{\zeta}(g_{\chi}) q^{1/2-s}$.
\end{cor}

\subsection{Langlands parameters for simple supercuspidals of $\GL_n(F)$}\label{langlandsgln}
We now wish to predict the Langlands parameter $\phi_{\zeta}$ associated to $\pi_{\zeta}$.  If $p \nmid n$, then automatically $\phi_{\zeta} = \Ind_{W_E}^{W_F}(\xi_{\zeta})$ for some tamely ramified
extension $E/F$ of degree $n$, and some character $\xi_{\zeta}$ of $E^{\times}$ (see \cite{BH10}). Here, $\xi_{\zeta}$ is viewed as a character of $W_E$ via local class field theory. In this section, we give a prediction for $\xi_{\zeta}$.  In \S\ref{computationepsilon} and \S\ref{epsilontwists}, we will prove that this prediction is correct.

Let $\varpi_E$ be an $n^{\mathrm{th}}$ root of $\varpi$, and set $E = F(\varpi_E)$.  $\xi_{\zeta}$ is then a character of $E^{\times} = \langle \varpi_E \rangle \times k_E^{\times} \times (1 + \mathfrak{p}_E)$.  Since $E/F$ is totally ramified, $k_E = k_F$.

There is a natural way of defining $\xi_{\zeta}$.  Relative to the basis
$$\varpi_E^{n-1}, \varpi_E^{n-2}, \cdots, \varpi_E, 1$$
of $E/F$, we have an embedding $$\iota : E^{\times} \hookrightarrow \GL_n(F).$$ Set $\xi_{\zeta}(w) = \chi_{\zeta}(\iota(w)) \ \forall w \in 1 + \mathfrak{p}_E$.

Moreover, we define $$\xi_{\zeta}|_{k_F^{\times}} = \omega|_{k_F^{\times}} \otimes (\varkappa_{E/F}|_{k_F^{\times}})^{-1}$$ (see \cite[Proposition 29.2]{BH06}), where $\varkappa_{E/F} = \mathrm{det}(\Ind_{W_E}^{W_F}(1_E))$, and $1_E$ denotes the trivial character of $W_E$.  We are forced to define $\xi_{\zeta}|_{k_F^{\times}}$ in this way because of the central character condition in the local Langlands correspondence for $\GL_n(F)$ (see \cite[p. 2]{HT01}).

Finally, we set $$\xi_{\zeta}(\varpi_E) = \chi_{\zeta}(g_{\chi}) \lambda_{E/F}(\psi)^{-1},$$ where $\lambda_{E/F}(\psi)$ is the Langlands constant (see \eqref{Langlandsconstant}).  This definition of $\xi_{\zeta}(\varpi_E)$ is forced upon us by the required matching of standard epsilon factors (see Corollary \ref{matchingstandardepsilonfactors}).
\subsection{Standard $\epsilon$-factors of Langlands parameters}\label{computationepsilon}

In this section, we compute $\epsilon(s, \Ind_{W_E}^{W_F}(\xi_{\zeta}), \psi)$, where $F, E, \xi_{\zeta}$ are as in \S\ref{langlandsgln}.  We recall that
\begin{equation}\label{epsilonfactorinduced}
\frac{\epsilon(s, \Ind_{W_E}^{W_F}(\xi_{\zeta}), \psi)}{\epsilon(s, \xi_{\zeta}, \psi_E)} = \frac{ \epsilon(s, \Ind_{W_E}^{W_F}(1_E), \psi)}{\epsilon(s, 1_E, \psi_E)},
\end{equation}
where $1_E$ denotes the trivial character of $W_E$, and where $\psi_E = \psi \circ \mathrm{Tr}_{E/F}$ \cite[\S29.4]{BH06}.  We also recall the Langlands constant (see \cite[30.4]{BH06})
\begin{equation}\label{Langlandsconstant}
\lambda_{E/F}(\psi) := \frac{\epsilon(s, \Ind_{W_E}^{W_F}(1_E), \psi)}{\epsilon(s, 1_E, \psi_E)}.
\end{equation}

Let $K$ be a $p$-adic field, with $\mathfrak{o}_K$ its ring of integers and $\mathfrak{p}_K$ the associated maximal ideal.  Let $\lambda$ be a ramified character of $K^{\times}$ of level $n \geq 0$, and suppose that $\Psi \in \widehat{K}$ has level one.  Let $c \in K$ satisfy $val_K(c) = -n$.  Let $\lambda^{\vee}$ denote the contragredient of $\lambda$.

\begin{defn}
We define the \emph{Gauss sum} of $\lambda$, relative to $\Psi$, to be
$$\tau(\lambda, \Psi) = \displaystyle\sum_{x \in \mathfrak{o}_K^{\times} / (1 + \mathfrak{p}_K^{n+1})} \lambda^{\vee}(cx) \Psi(cx)$$
\end{defn}

\begin{prop}\label{Gausssum}
$\tau(\xi_{\zeta}, \psi_E) = \xi_{\zeta}(\varpi_E) q$.
\end{prop}

\proof

Note that $\xi_{\zeta}$ has level one.  Therefore,
\begin{align*}
& \tau(\xi_{\zeta}, \psi_E) \\
= & \displaystyle\sum_{x \in \mathfrak{o}_E^{\times} / (1 + \mathfrak{p}_E^2)} \xi_{\zeta}^{-1}(\varpi_E^{-1} x) \psi(\mathrm{Tr}_{E/F}(\varpi_E^{-1} x))\\
= \ & \xi_{\zeta}(\varpi_E) \displaystyle\sum_{x \in \mathfrak{o}_E^{\times} / (1 + \mathfrak{p}_E^2)} \xi_{\zeta}(x)^{-1} \psi(\mathrm{Tr}_{E/F}(\varpi_E^{-1} x)).
\end{align*}
Let us identify $\mathfrak{o}_E^{\times} / (1 + \mathfrak{p}_E^2)$ with the elements $a_0 + a_1 \varpi_E$, with $a_i \in k_F$ and $a_0 \neq 0$.  If $x = a_0 + a_1 \varpi_E$, then one calculates that $\mathrm{Tr}_{E/F}(\varpi_E^{-1} x) = n a_1$.  Thus,
\begin{align*}
& \xi_{\zeta}(\varpi_E) \displaystyle\sum_{x \in \mathfrak{o}_E^{\times} / (1 + \mathfrak{p}_E^2)} \xi_{\zeta}(x)^{-1} \psi(\mathrm{Tr}_{E/F}(\varpi_E^{-1} x))\\
= \ & \xi_{\zeta}(\varpi_E) \displaystyle\sum_{\substack{a_0 \in k_F^{\times}\\ a_1 \in k_F}} \xi_{\zeta}(a_0 + a_1 \varpi_E)^{-1} \psi(n a_1)\\
= \ & \xi_{\zeta}(\varpi_E) \displaystyle\sum_{\substack{a_0 \in k_F^{\times}\\ a_1 \in k_F}} \xi_{\zeta}(a_0)^{-1} \xi_{\zeta}(1 + \frac{a_1}{a_0} \varpi_E)^{-1} \psi(n a_1)\\
= \ & \xi_{\zeta}(\varpi_E) \displaystyle\sum_{a_0 \in k_F^{\times}} \xi_{\zeta}(a_0)^{-1} \displaystyle\sum_{a_1 \in k_F} \xi_{\zeta}(1 + \frac{a_1}{a_0} \varpi_E)^{-1} \psi(n a_1).
\end{align*}
One can compute that if $x = 1 + c_1 \varpi_E + c_2 \varpi_E^2 + \cdots \in 1 + \mathfrak{p}_E$, then $\xi_{\zeta}(x) = \psi(n c_1)$, by definition of $\xi_{\zeta}|_{1 + \mathfrak{p}_E}$.  Therefore, we obtain
\begin{align*}
& \xi_{\zeta}(\varpi_E) \displaystyle\sum_{a_0 \in k_F^{\times}} \xi_{\zeta}(a_0)^{-1} \displaystyle\sum_{a_1 \in k_F} \psi(\frac{n a_1}{a_0})^{-1} \psi(n a_1)\\
= \ & \xi_{\zeta}(\varpi_E) \displaystyle\sum_{a_0 \in k_F^{\times}} \xi_{\zeta}(a_0)^{-1} \displaystyle\sum_{a_1 \in k_F} \psi(-\frac{n a_1}{a_0}) \psi(n a_1).
\end{align*}

Since the sum of a nontrivial character over a group is zero, the inner sum vanishes unless $a_0 = 1$.  This gives us the result.
\qed

Together with formula \cite[23.6.2]{BH06}, we now have

\begin{cor}\label{epsilonchi}
$\epsilon(s, \xi_{\zeta}, \psi_E) = \xi_{\zeta}(\varpi_E) q^{\frac{1}{2} - s}$.
\end{cor}

We therefore obtain

\begin{cor}\label{matchingstandardepsilonfactors}
$\epsilon(s, \phi_{\zeta}, \psi) = \epsilon(s,\pi_{\zeta}, \psi)$.
\end{cor}

\proof
By definition of $\lambda_{E/F}(\psi)$ and Corollary \ref{epsilonchi}, we have $$\epsilon(s, \phi_{\zeta}, \psi) = \xi_{\zeta}(\varpi_E) q^{\frac{1}{2} - s} \lambda_{E/F}(\psi).$$  By definition of $\xi_{\zeta}(\varpi_E)$, we have $$\xi_{\zeta}(\varpi_E) q^{\frac{1}{2} - s} \lambda_{E/F}(\psi) = \chi_{\zeta}(g_{\chi}) q^{\frac{1}{2} - s},$$ which equals $\epsilon(s,\pi_{\zeta}, \psi)$, by Corollary \ref{epsilonfactorgln}.
\qed

\subsection{Epsilon factors of one-dimensional twists}\label{epsilontwists}
By Corollary \ref{matchingstandardepsilonfactors}, we have matched up the standard epsilon factors of $\phi_{\zeta}$ with $\pi_{\zeta}$.  As is known, the standard epsilon factors are not enough to determine the local Langlands correspondence for $\GL_n(F)$.  Indeed, one can see from the computations in Proposition \ref{Gausssum} that $\tau(\xi_{\zeta}, \psi_E)$ had no dependence on $\xi_{\zeta}|_{k_F^{\times}}$.

In this section, we will compute $\epsilon(s, \pi_{\zeta} \times \lambda, \psi)$ and $\epsilon(s, \phi_{\zeta} \otimes \lambda, \psi)$, with $\lambda$ a tamely ramified character of $F^{\times}$.

\begin{lem}\label{twistedgammarepresentation}
Let $\lambda$ be a tamely ramified character of $F^{\times}$.  Then $\epsilon(s, \pi_{\zeta} \times \lambda, \psi) = \lambda(-1)^{n-1} \lambda(\varpi) \epsilon(s, \pi_{\zeta}, \psi)$.
\end{lem}

\proof
Setting $W' = \lambda$, we get similarly as in the proof of Theorem \ref{zetaintegral} that
\begin{align*}
& \widetilde{\Psi}(1-s;\rho(w_{n,m}) \widetilde{W}, \widetilde{W}')\\
= \ & \chi(g_{\chi}) \int_{\frac{1}{\varpi} + \mathfrak{o}} \int_{\mathfrak{o}} \int_{\mathfrak{o}} ... \int_{\mathfrak{o}}  dx_1 dx_2 \cdots dx_{n-2} \lambda(h)^{-1} |h|^{(1-s)-(n-1)/2} dh
\\
= \ & \chi(g_{\chi}) q^{(n-2)/2} \int_{\frac{1}{\varpi} + \mathfrak{o}} \lambda(h^{-1}) |h|^{(1-s)-(n-1)/2} dh\\
= \ & \chi(g_{\chi}) q^{(n-2)/2}  q^{(1-s)-(n-1)/2} \int_{\frac{1}{\varpi} + \mathfrak{o}} \lambda(h^{-1}) dh\\
= \ & \lambda(\varpi) \chi(g_{\chi})  q^{-s-1/2},
\end{align*}
since
$$\int_{\frac{1}{\varpi} + \mathfrak{o}} \lambda(h^{-1}) dh = \int_{1 + \mathfrak{p}} \lambda(\varpi h^{-1}) dh = \lambda(\varpi) vol(1 + \mathfrak{p}),$$
by translation invariance, and since $\lambda$ is tamely ramified.

Moreover, similarly as in Theorem \ref{otherzetaintegral},
\begin{align*}
& \Psi(s;W, W') \\
= & \int_{F^{\times}} W \left( \begin{array}{cc}
h & 0\\
0 & I_{n-1}
\end{array} \right) \lambda(h) |h|^{s-(n-1)/2} dh\\
= & \int_{1 + \mathfrak{p}} W \left( \begin{array}{cc}
h & 0\\
0 & I_{n-1}
\end{array} \right) \lambda(h) |h|^{s-(n-1)/2} dh = vol(1 + \mathfrak{p}).
\end{align*}
since $\lambda$ is tamely ramified.  Therefore, because of the presence of $\lambda(-1)^{n-1}$ in Theorem \ref{gammafactor}, this completes the proof of this lemma.
\qed

\begin{lem}\label{twistedgammaparameter}
Let $\lambda$ be a tamely ramified character of $F^{\times}$.  Then $\epsilon(s, \phi_{\zeta} \otimes \lambda, \psi) = \lambda(-1)^{n-1} \lambda(\varpi) \epsilon(s, \phi_{\zeta}, \psi)$.
\end{lem}

\proof
We first note that $\Ind_{W_E}^{W_F}(\xi_{\zeta}) \otimes \lambda = \Ind_{W_E}^{W_F}(\xi_{\zeta} \otimes \lambda_E)$, where $\lambda_E = \lambda \circ N_{E/F}$.  Then as in the proof of Proposition \ref{Gausssum},
\begin{align*}
& \tau(\xi_{\zeta} \otimes \lambda_E, \psi) \\
= & \displaystyle\sum_{x \in \mathfrak{o}_E^{\times} / (1 + \mathfrak{p}_E^2)} (\xi_{\zeta} \lambda_E)^{-1}(\varpi_E^{-1} x) \psi(\mathrm{Tr}_{E/F}(\varpi_E^{-1} x))\\
= \ & \xi_{\zeta}(\varpi_E) \lambda_E(\varpi_E) \displaystyle\sum_{\substack{a_0 \in k_F^{\times}\\ a_1 \in k_F}} (\xi_{\zeta} \lambda_E)(a_0 + a_1 \varpi_E)^{-1} \psi(n a_1)\\
= \ & \xi_{\zeta}(\varpi_E) \lambda_E(\varpi_E) \displaystyle\sum_{a_0 \in k_F^{\times}} (\xi_{\zeta} \lambda_E)(a_0)^{-1} \displaystyle\sum_{a_1 \in k_F} (\xi_{\zeta} \lambda_E)(1 + \frac{a_1}{a_0} \varpi_E)^{-1} \psi(n a_1).
\end{align*}
Since $\lambda$ is tamely ramified, we obtain
$$\xi_{\zeta}(\varpi_E) \lambda_E(\varpi_E) \displaystyle\sum_{a_0 \in k_F^{\times}} (\xi_{\zeta} \lambda_E)(a_0)^{-1} \displaystyle\sum_{a_1 \in k_F} \xi_{\zeta}(1 + \frac{a_1}{a_0} \varpi_E)^{-1} \psi(n a_1),$$
which, by the same argument as in the end of the proof of Proposition \ref{Gausssum}, equals $\lambda((-1)^{n-1} \varpi) \epsilon(s, \phi_{\zeta}, \psi)$.
\qed

\begin{cor}
Let $\lambda$ be a tamely ramified character of $F^{\times}$.  Then $\epsilon(s, \pi_{\zeta} \times \lambda, \psi) = \epsilon(s, \phi_{\zeta} \otimes \lambda, \psi)$.
\end{cor}

\begin{rmk}\label{simplesupercuspidaldetermined} $ $
\begin{enumerate}
\item One can see from \eqref{epsilonfactorinduced}, Corollary \ref{epsilonchi} and Lemma \ref{twistedgammaparameter} that after twisting by tamely ramified characters, we might not be able to distinguish between two Langlands parameters $\Ind_{W_E}^{W_F}(\chi)$ and $\Ind_{W_E}^{W_F}(\chi')$, if $\chi$ and $\chi'$ differ only on $k_F^{\times}$.
\item Recall that one of the conditions of the local Langlands correspondence for $\GL_n(F)$ is a matching between the central character of a representation of $\GL_n(F)$, and the determinant of its associated Langlands parameter (see \cite[p. 2]{HT01}). We have parameterized simple supercuspidal representations by a choice of uniformizer $\varpi$, a choice of $\zeta$, and a choice of central character $\omega$.  Corollary \ref{epsilonfactorgln} implies that the standard epsilon factor $\epsilon(s, \pi_{\zeta}, \psi)$ determines $\zeta$.  Lemma \ref{twistedgammarepresentation} implies that the twisted epsilon factors $\epsilon(s, \pi_{\zeta} \otimes \lambda, \psi)$ determine both $\zeta$ and $\varpi$, after varying $\lambda$ over all tamely ramified characters.  After these determinations, we only have to choose a central character.
Note that (see \cite[Proposition 29.2]{BH06})
$$
\det(\phi_{\zeta})=\det(\Ind_{W_E}^{W_F}(\xi_{\zeta}))
=\xi_{\zeta}|_{F^{\times}} \otimes \varkappa_{E/F}.
$$
By the standard properties of $\lambda_{E/F}(\psi)$ and $\varkappa_{E/F}$ (see \cite{M86} and \cite{BF83}), one can show that $\lambda_{E/F}(\psi)^n = \varkappa_{E/F}(\varpi)$ when $E/F$ is a totally tamely ramified extension of degree $n$.  It is not difficult to then see that after imposing $\xi_{\zeta}|_{k_F^{\times}} = \omega|_{k_F^{\times}} \otimes (\varkappa_{E/F}|_{k_F^{\times}})^{-1}$ on our Langlands parameter $\phi_{\zeta}$, the only simple supercuspidal representation $\pi$ of $\GL_n(F)$ that satisfies both
$$\gamma(s, \pi \times \lambda, \psi) = \gamma(s, \phi_{\zeta} \otimes \lambda, \psi) \ \ \mathrm{and}$$
$$\omega_{\pi} = \det(\phi_{\zeta})$$
for all tamely ramified characters $\lambda$ of $F^{\times}$, is $\pi_{\zeta}$.  Here, $\omega_{\pi}$ denotes the central character of $\pi$.
\end{enumerate}
\end{rmk}

\begin{thm}\label{maintheorem1}
The assignment $\pi_{\zeta} \mapsto \phi_{\zeta}$ is the local Langlands correspondence for simple supercuspidal representations of $\GL_n(F)$, when $p \nmid n$.
\end{thm}
\proof
The existence of the local Langlands correspondence for $\GL_n(F)$, due to Harris/Taylor and Henniart (see \cite{HT01, H00}) guarantees that there exists a supercuspidal representation $\pi'$ such that all of its twisted epsilon factors agree with the corresponding twisted epsilon factors of $\phi_{\zeta}$.
In particular, $\pi'$ and $\pi_{\zeta}$  have the same standard epsilon factors, hence the same conductors.
By \cite[Theorem 3.1]{LR03}, $\pi'$ must have depth $\frac{1}{n}$.  By \cite{BH13}, the simple supercuspidal representations of $\GL_n(F)$ are precisely the supercuspidal representations of $\GL_n(F)$ of depth $\frac{1}{n}$.  Therefore, $\pi'$ must be simple supercuspidal.  Therefore, by Remark \ref{simplesupercuspidaldetermined}, it must be that $\pi' = \pi_{\zeta}$.
\qed

\section{Epipelagic supercuspidal representations of $\GL_n(F)$}\label{epipelagicissimple}

 In this section, we begin by generalizing the construction of epipelagic supercuspidal representations of Reeder and Yu (see \cite{RY13}) to $\GL_n(F)$ (see \S\ref{constructionepipelagic}).   In \S\ref{stablefunctionals} and \S\ref{nonbarycenters}, we show that these constructed epipelagic supercuspidals of $\GL_n(F)$ can only come from a barycenter of an alcove in the building of $\GL_n(F)$.
 We conclude (see Remark \ref{epipelagicisnotsimple}) that the construction of Reeder and Yu does not necessarily exhaust all epipelagic supercuspidal representations of semisimple $p$-adic groups.

\subsection{Construction of epipelagic representations of $\GL_n(F)$}\label{constructionepipelagic}

In this section, we give an explicit construction of a class of epipelagic supercuspidal representations following the ideas in \cite{RY13}.

Let $\RG=\GL_n$, $\mathrm{T}$ the split maximal torus, and $\RZ$ the center.  We also let
$G=\GL_n(F)$, $T=\mathrm{T}(F)$, and $Z=\RZ(F)$.  Set $F^{\mathrm{un}}$ to be the maximal unramified extension of $F$, and let $Frob: \RG(F^{\mathrm{un}}) \rightarrow \RG(F^{\mathrm{un}})$ be the Frobenius action, arising from the given $F$-structure on $\RG$.

Let $\CB(\RG,F)$ be the Bruhat-Tits building of $G$, and let $\CA(\RT, F)=X_*(T) \otimes \BR$ be the apartment of $T$ in $\CB(\RG,F)$.  Let $\CB(\RG,F^{\mathrm{un}})$ be the Bruhat-Tits building of $\RG(F^{\mathrm{un}})$, and let $\CA(\RT, F^{\mathrm{un}})=X_*(\RT(F^{\mathrm{un}})) \otimes \BR$ be the apartment of $\RT(F^{\mathrm{un}})$ in $\CB(\RG,F^{\mathrm{un}})$. Then we identify
$$\CB(\RG,F)=\CB(\RG,F^{\mathrm{un}})^{Frob}, \text{ and } \CA(\RT, F)=\CA(\RT, F^{\mathrm{un}})^{Frob}.$$

Let $\Phi$ be the roots of $T$ in $G$. Let $\Psi=\{\phi + n | \phi \in \Phi, n \in \BZ\}$ be the affine roots of $T$ in $G$, where $(\phi+n)(\alpha)=\phi(\alpha)+n$, for $\alpha \in \CA(\RT, F)$. If $\psi=\phi + n \in \Psi$, then $\phi$ is called the \emph{gradient} of $\psi$ and denoted by $\dot{\psi}$.

Fix a pinning for $G$.  For any $x \in \CA(\RT, F)$, and any real number $r \geq 0$, let
\begin{align*}
\RT(F)_{r} & := \RT(1+\frak{p}_F^{\lceil r \rceil}),\\
\RT(F)_{r^+} & := \RT(1+\frak{p}_F^{\lfloor r \rfloor+1}),\\
\RG(F)_{x,r} & := \langle \RT(F)_r, X_{\alpha}(\frak{p}_F^{-\lfloor \alpha(x)-r \rfloor}) | \alpha \in \Phi \rangle,\\
\RG(F)_{x,r^+} & := \langle \RT(F)_{r^+}, X_{\alpha}(\frak{p}_F^{1-\lceil \alpha(x)-r \rceil}) | \alpha \in \Phi \rangle
\end{align*}
be the filtration groups defined by Moy and Prasad \cite{MP94}, where $X_{\alpha}$ is the root group homomorphism associated to $\alpha$.

By Theorem 5.2 of \cite{MP94}, given any irreducible admissible representation $(\pi, V)$ of $G$,
there is a nonnegative rational number $r=\rho(\pi)$ with the property that $r$ is the minimal
number such that $V^{\RG(F)_{x,r^+}}$ is nonzero for some $x \in \CB(\mathrm{G},F)$. This number $r=\rho(\pi)$ is called the \emph{depth} of $\pi$.

For any given point $x \in \CA(\mathrm{T}, F)$, let $r(x)$ be the smallest positive value in the set $\{ \psi(x) | \psi \in \Psi\}$. An irreducible representation $(\pi,V)$ of $G$ is called \emph{epipelagic} if it has a vector fixed under the subgroup $\RG(F)_{x, r(x)^+}$ and has depth $r(x)$, for some $x \in \CA(\mathrm{T}, F)$.

Set
\begin{align*}
J_x & := Z \cdot \RG(F)_{x,r(x)},\\
J_x^+ & := Z \cdot \RG(F)_{x,r(x)^+},\\
\mathrm{V}_x & := \mathrm{V}_{x,r(x)} : = \RG(F^{\mathrm{un}})_{x,r(x)} / \RG(F^{\mathrm{un}})_{x,r(x)^+},\\
\check{\mathrm{V}}_x & := \check{\mathrm{V}}_{x,r(x)} :=
\Hom_{\overline{k_F}}(\mathrm{V}_{x,r(x)}, \overline{k_F}),\\
\RG_x & := \RG(F^{\mathrm{un}})_{x,0} / \RG(F^{\mathrm{un}})_{x,0^+}.
\end{align*}
Note that from the definition of $r(x)$, $G_{x,0^+} = G_{x,r(x)}$. Moreover, $\mathrm{G}_x$ is a connected reductive group
over $\overline{k_F}$, and $\mathrm{V}_x$ is a finite-dimensional vector space over $\overline{k_F}$, hence is abelian.
The conjugation action of $\RG(F^{\mathrm{un}})_{x,0}$ on $\RG(F^{\mathrm{un}})_{x,r(x)}$ induces an algebraic representation $N : \mathrm{G}_x \rightarrow \GL({\mathrm{V}}_x)$ of $\mathrm{G}_x$ on ${\mathrm{V}}_x$, and we denote by $\check{\mathrm{V}}_x$ its the dual representation.  We also more generally define $\mathrm{V}_{x,r} = \RG(F^{\mathrm{un}})_{x,r} / \RG(F^{\mathrm{un}})_{x,r^+}$ and $\check{\mathrm{V}}_{x,r} = \Hom_{\overline{k_F}}(\mathrm{V}_{x,r}, \overline{k_F})$.

\begin{defn}
A functional $\lambda \in \check{\mathrm{V}}_{x,r}$ is called \emph{semistable} if its orbit under $\mathrm{G}_x$ does not contain zero in its closure, under the Zariski topology on $\check{\mathrm{V}}_{x,r}$.
\end{defn}

Fix a nontrivial additive character $\chi$ of $k_F$.
Let $\lambda \in \check{\mathrm{V}}_{x,r}(k_F)$, and set $\chi_{\lambda} := \chi \circ \lambda$. Then
$\chi_{\lambda}$ is a character of $\RG(F)_{x,r}$ which is trivial on $\RG(F)_{x,r^+}$.

\begin{rmk}\label{moyprasadsemistable}
In \cite{MP94}, Moy and Prasad proved the result that for every irreducible admissible positive-depth representation $\pi$ of $G$, there is a pair $(x,r) \in \CA(\RT, F) \times \mathbb{R}_{> 0}$ such that $\pi$ contains a character $\chi_{\lambda}$ of $\RG(F)_{x,r} / \RG(F)_{x,r^+}$ such that $\lambda \in \check{\mathrm{V}}_{x,r}(k_F)$ is semistable.
\end{rmk}

\begin{defn}\label{def3}
A functional $\lambda \in \check{\mathrm{V}}_x$ is called \emph{stable} if the following two conditions hold:

(1) the orbit $\mathrm{G}_x \cdot \lambda$ is Zariski-closed in $\check{\mathrm{V}}_x$,

(2) the stabilizer of $\lambda$ in $\mathrm{G}_x$ is finite, modulo $\mathrm{ker}(N)$.
\end{defn}

Let $\RG(F^{\mathrm{un}})_x := \{ g\in \RG(F^{\mathrm{un}}) | g \cdot x = x \}$. Then $\RG(F^{\mathrm{un}})_x$ contains
$\RG(F^{\mathrm{un}})_{x,0}$ with finite index. The action of $\RG(F^{\mathrm{un}})_{x,0}$ on
$\mathrm{V}_x$ and $\check{\mathrm{V}}_x$ extend to $\RG(F^{\mathrm{un}})_x$ which preserves the set of $\RG_x$-stable functionals in $\check{\mathrm{V}}_x$.  Note that since $\RG = \GL_n$, $\RG(F^{\mathrm{un}})_x = \RG(F^{\mathrm{un}})_{x,0}$.

Let $H_x=Z \cdot \RG(F)_x$.
Let $\lambda \in \check{\mathrm{V}}_x(k_F)$ be a $k_F$-rational stable functional with stabilizer $H_{x,\lambda}$ in $H_x$.  Then $H_{x,\lambda}=Z \cdot Stab_{\RG(F)_x}(\lambda)$.

We will show in Theorem \ref{main3} that if $x$ is a nonbarycenter point of the building, then $\check{\mathrm{V}}_{x,r(x)}$ contains no semistable functionals.  Therefore, in order to understand the epipelagic supercuspidal representation theory of $\GL_n$, it suffices to consider only barycenters.

Suppose that $x$ is a barycenter of a facet.  Let $\RZ_x$ be the diagonal embedding of $\ol{k_F}^{\times}$ in $\mathrm{G}_x$.  If the facet is a proper facet, we will show in Theorem \ref{main2} that $\mathrm{ker}(N) = \RZ_x$. If the facet is an alcove, it is easy to see that $\mathrm{ker}(N) = \RZ_x$.

Fix a character $\omega$ of $Z$ which is trivial on $\RZ(1+\frak{p}) := Z \cap \mathrm{T}(F)_1$. Note that $Z \cap \RG(F)_{x,r(x)} = \RZ(1+\frak{p})$, and that $\chi_{\lambda}$ is trivial on $\RZ(1+\frak{p})$. Therefore, we can define a character $\chi_{\lambda}^{\omega}$ on $J_x$ as follows:
$$\chi_{\lambda}^{\omega}(zg)=\omega(z) \chi_{\lambda}(g),$$
for $z \in Z$ and $g \in \RG(F)_{x,r(x)}$.

Let
$$\pi_x(\lambda, \omega) := \mathrm{ind}_{J_x}^G \chi_{\lambda}^{\omega},$$
be the compactly induced representation, and set $$\CH_{x,\lambda}:=End_{H_{x,\lambda}}(\mathrm{ind}_{J_x}^{H_{x,\lambda}} \chi_{\lambda}^{\omega}).$$
Then as in Section 2.1 of \cite{RY13}, from Mackey theory for finite groups, the intertwining algebra $\CH_{x,\lambda}$ has dimension equal to $\lvert H_{x,\lambda} / J_x \rvert$, and there is a bijection, denoted by $\rho \mapsto \chi^{\omega}_{\lambda, \rho}$, from the set of irreducible $\CH_{x,\lambda}$-modules to the set $\mathrm{Irr}(\CH_{x,\lambda})$ of irreducible constituents of $\mathrm{ind}_{J_x}^{H_{x,\lambda}} \chi_{\lambda}^{\omega}$ such that
\begin{equation}\label{chirho}
\mathrm{ind}_{J_x}^{H_{x,\lambda}} \chi_{\lambda}^{\omega} = \bigoplus_{\rho \in \mathrm{Irr}(\CH_{x,\lambda})} (\dim \rho) \chi^{\omega}_{\lambda, \rho}.
\end{equation}

A point $x \in \CA(\RT, F^{\mathrm{un}})$ is called \emph{rational} if $\psi(x) \in \BQ$ for any $\psi \in \Psi$.
Given two rational points $x,y \in \CA(\RT, F)$ and a rational number $r > 0$, let $\RV_{x,y,r}$ be the image of $\RG(F^{\mathrm{un}})_{y,r^+} \cap \RG(F^{\mathrm{un}})_{x,r}$ in $\RV_{x,r}$.
Let $\Psi_{x,r}=\{\psi \in \Psi|\psi(x)=r \text{ or } r-1\}$.
Then from the definition of $\RV_{x,r}$, it is easy to see that
the nonzero weights of $\RT$ in $\RV_{x,r}$ are $\{\dot{\psi}|\psi\in \Psi_{x,r}\}$. Moreover, $\RV_{x,y,r}$ is a $\RT$-stable subspace
of $\RV_{x,r}$ with weight decomposition
$$\RV_{x,y,r}=\bigoplus_{\psi \in \Psi_{x,r}, \psi(y)>0} \RV_{x,r}(\dot{\psi}).$$

\begin{lem}\label{lem2.3RY13}
Let $x$ be a barycenter of a facet.  Let $\lambda \in \check{\RV}_{x,r}$ be a stable functional. If $\lambda$ vanishes identically on $\RV_{x,y,r}$, then $x-y \in X_*(Z) \otimes \mathbb{R}$.
\end{lem}

\begin{proof}
The proof of this lemma is almost verbatim of the proof \cite[Lemma 2.3]{RY13}, but with a simple modification.
\end{proof}

Following the idea of Reeder and Yu \cite{RY13}, we have the following analogous result for epipelagic supercuspidal representations of $\GL_n(F)$.

\begin{thm}\label{thmRY13}
Let $x \in \CA(\mathrm{T}, F)$ be a barycenter of a facet. Assume that $\lambda \in \check{\mathrm{V}}_x$ is a $k_F$-rational stable functional. Then $\pi_x(\lambda, \omega)$ has a finite direct sum decomposition
$$\pi_x(\lambda, \omega) = \bigoplus_{\rho \in \mathrm{Irr}(\CH_{x,\lambda})} (\dim \rho) \cdot \pi_x(\lambda,\omega, \rho),$$
where $\pi_x(\lambda, \omega,\rho) := \mathrm{ind}_{H_{x,\lambda}}^{G} \chi_{\lambda, \rho}^{\omega}$ is an irreducible supercuspidal representation of $\GL_n(F)$ with central character $\omega$ compactly induced from $\chi_{\lambda, \rho}^{\omega}$ which is as in \eqref{chirho}, for each $\rho$.
Moreover, if $\rho \ncong \rho'$, then $\pi_x(\lambda, \omega,\rho)
\ncong \pi_x(\lambda, \omega,\rho')$.  These $\pi_x(\lambda, \omega,\rho)$'s are epipelagic supercuspidal representations of $\GL_n(F)$.
\end{thm}

\begin{proof}
Using Lemma \ref{lem2.3RY13}, it is not difficult to see that the proof of \cite[Proposition 2.4]{RY13} goes through in our setting.
\end{proof}

\subsection{Stable functionals and barycenters of proper facets}\label{stablefunctionals}

In this section, we show that for $\mathrm{G} = \GL_n$, if $x$ is a barycenter in $\CB(\RG,F)$ such that $\check{\RV}_{x}$ has a stable functional, then $x$ is a barycenter of an alcove.  This implies that the epipelagic supercuspidal representations of $G$ obtained in Theorem \ref{thmRY13} cannot be constructed from barycenters of proper facets.

Recall that an {\it alcove} in $\CA(\mathrm{T}, F)$ is a connected component of the set of points in $\CA(\mathrm{T}, F)$ on which no affine root vanishes. Let $\psi_i=e_i-e_{i+1}$, for $i=1,\ldots, n-1$, and let $\psi_n=1-(e_1-e_n)$. These are the \emph{standard simple affine roots} of $\mathrm{G}$.  Then
the zero locus of $\{\psi_i : 1 \leq i \leq n \}$ bounds an alcove, which we denote by $C$.  Let $\overline{C}$ denote the closure of $C$.

First, we show the following theorem.

\begin{thm}\label{main2}
If $\overline{x}$ is a barycenter of a proper facet $\mathcal{F}$ of $\overline{C}$, then $\check{\RV}_{\overline{x}}$ has no stable functionals.
\end{thm}

\begin{proof}
Fix a point $\overline{x} \in \overline{C}$, which is a barycenter of a fixed proper facet $\mathcal{F}$ of $\overline{C}$. First, we will show that the dimension of $\mathrm{G}_{\overline{x}}$ is greater than or equal to that of $\mathrm{V}_{\overline{x}}$, hence also to that of $\check{\mathrm{V}}_{\overline{x}}$. In the rest of the proof, we identify $\mathrm{V}_{\overline{x}}$ with
$\check{\mathrm{V}}_{\overline{x}}$ via the standard trace pairing.

Note that for any $\alpha \in \Phi$, $\lceil \alpha(\overline{x})-r \rceil - \lfloor \alpha(\overline{x})-r \rfloor = 1$ if and only if $\alpha(\overline{x})-r$ is not an integer.
Therefore,
\begin{eqnarray}
\mathrm{dim}(\mathrm{G}_{\overline{x}}) = \# \{\alpha \in \Phi : \alpha(\overline{x}) \in \mathbb{Z} \} + n, \label{dimgx} \\
\mathrm{dim}(\mathrm{V}_{\overline{x}}) = \# \{\alpha \in \Phi : \alpha(\overline{x}) - r(\overline{x}) \in \mathbb{Z} \}. \label{dimvx}
\end{eqnarray}

Write $\overline{x} =(\overline{x}_1, \overline{x}_2, \ldots, \overline{x}_n) \in \BR^n$.  Then $\psi_i(\overline{x})=\overline{x}_i-\overline{x}_{i+1}$ for $1 \leq i \leq n-1$,
and $\psi_n(\overline{x})=1-(\overline{x}_1-\overline{x}_n)$.

Assume that $\mathcal{F}$ is the zero locus of the following set of simple affine roots:
\begin{align}\label{PsiF}
\Psi_{\mathcal{F}} := \left\{\psi_j, t \psi_n : 1 + \sum_{l=1}^{i-1}m_l \leq j \leq m_i-1  + \sum_{l=1}^{i-1}m_l, i = 1, 2, ..., k \right\},
\end{align}
where $t \in \{0,1\}$, $k \in \{1, 2, ..., n \}$, $m_i \geq 1$ for all $i = 1,2,...,k$, and $\sum_{i=1}^k m_i=n$.  Any facet $\mathcal{F}$ can be expressed in this form.  We note that if $t=0$, then there exists $i \in \{1, 2, \ldots, k \}$ such that $m_i >1$.  We also note that if $t = 1$ and $k = 1$, then $\mathcal{F} = \emptyset$.

We need to consider the cases $t = 0$ and $t = 1$ separately.

Case (1): Suppose that $t=0$.  Note that for $\overline{x}=(\overline{x}_1, \overline{x}_2, \ldots, \overline{x}_n)$, we always have
$\sum_{i=1}^n \psi_i(\overline{x})=1$.
Since $\overline{x}$ is a barycenter determined by $k$ simple affine roots, we have $r(\overline{x})=1/k$.

By solving the following equations:
\begin{align*}
& \psi(\overline{x})=0,  \text{ for } \ \psi \in \Psi_{\mathcal{F}} \\
& \psi(\overline{x})=1/k, \text{ for } \ \psi \notin \Psi_{\mathcal{F}},
\end{align*}
we can see that we may set
\begin{equation}\label{barycenter1}
\overline{x} = (\overline{x}_1, \ldots, \overline{x}_{k}),
\end{equation}
where for $1 \leq i \leq k$,
\begin{align*}
\overline{x}_i & = \left(a-\frac{i}{k}, a-\frac{i}{k}, \ldots, a-\frac{i}{k} \right) \in \BR^{m_i}
\end{align*}
for some $a \in \mathbb{R}$.

By \eqref{dimgx}, one can see that $\dim(\mathrm{G}_{\overline{x}})=\sum_{i=1}^k m_i^2$.
Explicitly,
\begin{equation*}
\mathrm{G}_{\overline{x}} \cong \prod_{i=1}^k \mathrm{GL}_{m_i}(\ol{k_F}) \cong \{\diag(g_1, g_2, \ldots, g_k) : g_i \in \mathrm{GL}_{m_i}(\ol{k_F}), 1 \leq i \leq k\}.
\end{equation*}

By \eqref{dimvx}, one can see that $\dim(\mathrm{V}_{\overline{x}}) = \left(\sum_{i=1}^{k-1} m_im_{i+1} \right) + m_1m_k$.
Explicitly,
\begin{align*}
\begin{split}
\mathrm{V}_{\overline{x}} \cong \ & \{v(\overline{X}_1, \overline{X}_2, \ldots, \overline{X}_k) : \overline{X}_i \in \mathrm{M}_{m_i \times m_{i+1}}(\ol{k_F}), \\
& \ 1 \leq i \leq k-1, \overline{X}_k \in \mathrm{M}_{m_k \times m_1}(\ol{k_F}) \},
\end{split}
\end{align*}
where
\[
v(\overline{X}_1, \overline{X}_2, \ldots \overline{X}_k) := \begin{pmatrix}
0_{m_1} & \overline{X}_1 & 0 & \cdots & 0\\
0 & 0_{m_2} & \overline{X}_2 & \cdots & 0 \\
0 & 0 & 0_{m_3} & \ddots & 0\\
0 & 0 & 0 & \ddots & \overline{X}_{k-1}\\
\overline{X}_k & 0 & 0 & 0 & 0_{m_k}
\end{pmatrix},
\]
where $0_r$ denotes the zero $r \times r$ matrix.  We note that for any $g = \diag(g_1, g_2, \ldots, g_k) \in \mathrm{G}_{\overline{x}}$
and $v \in \mathrm{V}_{\overline{x}}$, $\mathrm{G}_{\overline{x}}$ acts on $\mathrm{V}_{\overline{x}}$ by
$$g \cdot v(\overline{X}_1, \overline{X}_2, \ldots, \overline{X}_k)=v(g_1\overline{X}_1g_2^{-1}, g_2\overline{X}_2g_3^{-1}, \ldots, g_k\overline{X}_kg_1^{-1}).$$

Since
\begin{align*}
& \sum_{i=1}^k m_i^2 - \left[ \left(\sum_{i=1}^{k-1} m_im_{i+1} \right) + m_1m_k \right]\\
= \ & \frac{1}{2}\left[\left(\sum_{i=1}^{k-1}(m_i-m_{i+1})^2 \right) + (m_1-m_k)^2\right],
\end{align*}
it is clear that
$\dim(\mathrm{G}_{\overline{x}}) \geq \dim(\mathrm{V}_{\overline{x}})$.
Moreover, $\dim(\mathrm{G}_{\overline{x}}) = \dim(\mathrm{V}_{\overline{x}})$ if and only if $m_1 = m_2 = \cdots = m_k$.

Case (2): Suppose that $t=1$.  If $k=1$, then $\mathcal{F} = \emptyset$.

Suppose that $k > 1$. Then $r(\overline{x}) = 1/(k-1)$, and we can see that we may set
\begin{equation}\label{barycenter2}
\overline{x} = (\overline{x}_1, \ldots, \overline{x}_{k}),
\end{equation}
where for $1 \leq i \leq k$,
\begin{align*}
\overline{x}_i & = \left(a-\frac{i}{k-1}, a-\frac{i}{k-1}, \ldots, a-\frac{i}{k-1}\right) \in \BR^{m_i}.
\end{align*}
for some $a \in \mathbb{R}$.

By \eqref{dimgx}, one can see that $\dim(\mathrm{G}_{\overline{x}})=(m_1+m_k)^2 + \sum_{i=2}^{k-1} m_i^2.$
Explicitly,
\begin{align*}
\begin{split}
\mathrm{G}_{\overline{x}}  \cong \mathrm{GL}_{m_1+m_k}(\ol{k_F}) \times \prod_{i=2}^{k-1} \mathrm{GL}_{m_i}(\ol{k_F}) \cong \{\nu(g_1, g_2, \ldots, g_{k+2}) : \\
g_i \in \mathrm{GL}_{m_i}(\ol{k_F}), 1 \leq i \leq k, g_{k+1} \in \mathrm{M}_{m_1 \times m_k}(\ol{k_F}), g_{k+2} \in \mathrm{M}_{m_k \times m_1}(\ol{k_F})\},
\end{split}
\end{align*}
where
\[
\nu(g_1, g_2, \ldots, g_{k+2}) :=
\begin{pmatrix}
g_1 & 0 & 0 & \cdots & g_{k+1}\\
0 & g_2 & 0 & \cdots & 0 \\
0 & 0 & g_3 & \ddots & 0\\
0 & 0 & 0 & \ddots & 0\\
g_{k+2} & 0 & 0 & 0 & g_k
\end{pmatrix}.
\]
By \eqref{dimvx}, one can see that $\dim(\mathrm{V}_{\overline{x}}) = \left(\sum_{i=1}^{k-1} m_im_{i+1} \right) + m_{k-1}m_1+m_km_2.$  Explicitly,
\begin{align*}
\begin{split}
\mathrm{V}_{\overline{x}} \cong \ & \{v(\overline{X}_1, \overline{X}_2, \ldots, \overline{X}_{k+1}) : \overline{X}_i \in \mathrm{M}_{m_i \times m_{i+1}}(\ol{k_F}), \\
& \ 1 \leq i \leq k-1, \overline{X}_k \in \mathrm{M}_{m_{k-1} \times m_1}(\ol{k_F}), \overline{X}_{k+1} \in \mathrm{M}_{m_{k} \times m_2}(\ol{k_F})\},
\end{split}
\end{align*}
where
\[
v(\overline{X}_1, \overline{X}_2, \ldots, \overline{X}_{k+1}) := \begin{pmatrix}
0_{m_1} & \overline{X}_1 & 0 & \cdots & 0\\
0 & 0_{m_2} & \overline{X}_2 & \cdots & 0 \\
0 & 0 & 0_{m_3} & \ddots & 0\\
\overline{X}_k & 0 & 0 & \ddots & \overline{X}_{k-1}\\
0 & \overline{X}_{k+1} & 0 & 0 & 0_{m_k}
\end{pmatrix}.
\]

Let $w = \begin{pmatrix}
I_{m_1} & 0 & 0\\
0 & 0 & I_{m_k}\\
0 & I_{\sum_{i=2}^{k-1}m_i} & 0
\end{pmatrix}$, where $I_r$ denotes the $r \times r$ identity matrix.
Conjugating $\mathrm{G}_{\overline{x}}$ and $\mathrm{V}_{\overline{x}}$ by $w$, we get
\begin{align*}
\begin{split}\mathrm{G}_{\overline{x}} & \cong \mathrm{GL}_{m_1+m_k}(\ol{k_F}) \times \prod_{i=2}^{k-1} \mathrm{GL}_{m_i}(\ol{k_F}) \cong \{\diag(g_1, g_2, \ldots, g_{k-1}) : \\
& g_1 \in \mathrm{GL}_{m_1+m_k}(\ol{k_F}), g_i \in \mathrm{GL}_{m_i}(\ol{k_F}), 2 \leq i \leq k-1\},
\end{split}
\end{align*}
\begin{align*}
\begin{split}
\mathrm{V}_{\overline{x}} \cong \ & \{v(\overline{X}_1, \overline{X}_2, \ldots, \overline{X}_{k-1}) : \overline{X}_1 \in \mathrm{M}_{(m_1+m_k) \times m_2}(\ol{k_F}), \\
& \ \overline{X}_i \in \mathrm{M}_{m_i \times m_{i+1}}(\ol{k_F}), 2 \leq i \leq k-2, \overline{X}_{k-1} \in \mathrm{M}_{m_{k-1} \times (m_1+m_k)}(\ol{k_F})\},
\end{split}
\end{align*}
where
\[
v(\overline{X}_1, \overline{X}_2, \ldots \overline{X}_{k-1}) := \begin{pmatrix}
0_{m_1+m_k} & \overline{X}_1 & 0 & \cdots & 0\\
0 & 0_{m_2} & \overline{X}_2 & \cdots & 0 \\
0 & 0 & 0_{m_3} & \ddots & 0\\
0 & 0 & 0 & \ddots & \overline{X}_{k-2}\\
\overline{X}_{k-1} & 0 & 0 & 0 & 0_{m_{k-1}}
\end{pmatrix}.
\]
Moreover, for any $g\in \mathrm{G}_{\overline{x}}$ and $v \in \mathrm{V}_{\overline{x}}$ as above,
$$g\cdot v(\overline{X}_1, \overline{X}_2, \ldots, \overline{X}_{k-1})=v(g_1\overline{X}_1g_2^{-1}, g_2\overline{X}_2g_3^{-1}, \ldots, g_{k-1}\overline{X}_{k-1}g_1^{-1}).$$

Since
\begin{align*}
& \ \left[(m_1+m_k)^2 + \left(\sum_{i=2}^{k-1} m_{i}^2\right)\right]\\
& \ -\left[\left(\sum_{i=2}^{k-2} m_im_{i+1} \right)+ m_{k-1}(m_1+m_k)+(m_1+m_k)m_2)\right]\\
= \ & \left[(m_1+m_k)^2 + \sum_{i=2}^{k-1} m_{i}^2\right] - \left[\left(\sum_{i=1}^{k-1} m_im_{i+1} \right)+ m_{k-1}m_1+m_km_2\right]\\
= \ & \frac{1}{2}\left[\left(\sum_{i=2}^{k-2}(m_i-m_{i+1})^2 \right) + (m_1+m_k-m_2)^2 + (m_1+m_k-m_{k-1})^2\right],
\end{align*}
it is easy to see that $\dim(\mathrm{G}_{\overline{x}}) \geq \dim(\mathrm{V}_{\overline{x}})$. Moreover, $\dim(\mathrm{G}_{\overline{x}}) = \dim(\mathrm{V}_{\overline{x}})$ if and only if
$m_2=m_3=\cdots=m_{k-1}=m_1+m_k$.

The analysis in Case (2) implies that it is enough to consider Case (1).  In other words, we have reduced ourselves to the following situation:
\begin{equation*}
\mathrm{G}_{\overline{x}} \cong \prod_{i=1}^k \mathrm{GL}_{m_i}(\ol{k_F}) \cong \{\diag(g_1, g_2, \ldots, g_k) : g_i \in \mathrm{GL}_{m_i}(\ol{k_F}), 1 \leq i \leq k\},
\end{equation*}
\begin{align}\label{Vx1}
\begin{split}
\mathrm{V}_{\overline{x}} \cong \ & \{v(\overline{X}_1, \overline{X}_2, \ldots, \overline{X}_k) : \overline{X}_i \in \mathrm{M}_{m_i \times m_{i+1}}(\ol{k_F}), \\
& \ 1 \leq i \leq k-1, \overline{X}_k \in \mathrm{M}_{m_k \times m_1}(\ol{k_F})\},
\end{split}
\end{align}
with
$$g \cdot v(\overline{X}_1, \overline{X}_2, \ldots, \overline{X}_k)=v(g_1\overline{X}_1g_2^{-1}, g_2\overline{X}_2g_3^{-1}, \ldots, g_k\overline{X}_kg_1^{-1})$$
for any $g \in \mathrm{G}_{\overline{x}}$ and $v \in \mathrm{V}_{\overline{x}}$ as above.
Moreover, $\dim(\mathrm{G}_{\overline{x}}) \geq \dim(\mathrm{V}_{\overline{x}})$, and
$\dim(\mathrm{G}_{\overline{x}}) = \dim(\mathrm{V}_{\overline{x}})$ if and only if $m_1 = m_2 = \cdots = m_k$. We note that since $\overline{x}$ is a barycenter of a proper facet and since $t = 0$,
there exists $i \in \{1, 2, \ldots, k \}$ such that $m_i > 1$.

Let $\lambda \in \mathrm{V}_{\overline{x}}$, and suppose that $\mathrm{G}_{\overline{x}} \cdot \lambda = \mathrm{V}_{\overline{x}}$.  Since $0 \in \mathrm{V}_{\overline{x}}$, we get $\mathrm{V}_{\overline{x}} = 0$, a contradiction.  Therefore any Zariski-closed orbit in $\mathrm{V}_{\overline{x}}$
will be a proper subset of $\mathrm{V}_{\overline{x}}$, hence with dimension
less than or equal to $\dim(\mathrm{V}_{\overline{x}})-1$.

We will now prove that there are no stable functionals in $\mathrm{V}_{\overline{x}}$ for the action of $\mathrm{G}_{\overline{x}}$.  To do this, we will repeatedly make use of the following dimension formula relating orbits and stabilizers.  If $\lambda \in \mathrm{V}_{\overline{x}}$ and $\mathrm{Stab}_{\mathrm{G}_{\overline{x}}}(\lambda)$ denotes the stabilizer in $\mathrm{G}_{\overline{x}}$ of $\lambda$, then
$$\mathrm{dim}(\mathrm{G}_{\overline{x}}) = \mathrm{dim}(\mathrm{Stab}_{\mathrm{G}_{\overline{x}}}(\lambda)) + \mathrm{dim}(\mathrm{G}_{\overline{x}} \cdot \lambda).$$
The action of $\mathrm{G}_{\overline{x}}$ on $\mathrm{V}_{\overline{x}}$ gives rise to a representation
$$N : \mathrm{G}_{\overline{x}} \rightarrow \GL(\mathrm{V}_{\overline{x}})$$
We first need to show that $\mathrm{ker}(N) = \RZ_x$, the latter of which is clearly one-dimensional.  It is clear that $\RZ_x \subset \mathrm{ker}(N)$.  Let $g \in \mathrm{G}_{\overline{x}}$ such that $$g \cdot v(\overline{X}_1, \overline{X}_2, \ldots, \overline{X}_k)= v(\overline{X}_1, \overline{X}_2, \ldots, \overline{X}_k)$$ for every $\lambda = v(\overline{X}_1, \overline{X}_2, \ldots, \overline{X}_k) \in \mathrm{V}_{\overline{x}}$.
Then $$v(g_1\overline{X}_1g_2^{-1}, g_2\overline{X}_2g_3^{-1}, \ldots, g_k\overline{X}_kg_1^{-1}) = v(\overline{X}_1, \overline{X}_2, \ldots, \overline{X}_k)$$ for every $\lambda = v(\overline{X}_1, \overline{X}_2, \ldots, \overline{X}_k) \in \mathrm{V}_{\overline{x}}$.  We claim that if $h_1 \in \GL_{s_1}(F)$ and $h_2 \in \GL_{s_2}(F)$ such that $h_1 Y h_2^{-1} = Y$ for all $Y \in \mathrm{M}_{s_1 \times s_2}(F)$, then $h_1, h_2$ are both central, and moreover each have the same element along the diagonal.

Indeed, if $s_1 = s_2$, then varying $Y \in M_{s_1 \times s_2}$ gives the result.  Suppose without loss of generality that $s_2 > s_1$.  By choosing $Y$ to be each of the following matrices
\begin{align*}
& \begin{pmatrix}
I_{s_1} & 0 & 0 & \ldots & 0 & 0
\end{pmatrix},
\begin{pmatrix}
0 & I_{s_1} & 0 & \ldots & 0 & 0
\end{pmatrix},\\
&
\begin{pmatrix}
0 & 0 & I_{s_1} & \ldots & 0 & 0
\end{pmatrix},
\ldots,
\begin{pmatrix}
0 & 0 & 0 & \ldots & I_{s_1} & 0
\end{pmatrix},
& \mathrm{and} \
\begin{pmatrix}
0 & I_r\\
0 & 0
\end{pmatrix},
\end{align*}
where $r$ is the remainder when one divides $s_2$ by $s_1$, one can see that $h_2$ is block diagonal, with each block equal to $h_1$ except for possibly the last block. The last block, of size $r \times r$, is equal to the upper left $r \times r$ block of $h_1$.  Now that $h_2$ is in this form, when one varies $Y$ across all matrices of the form
\begin{align*}
& \begin{pmatrix}
Y' & 0 & 0 & \ldots & 0 & 0
\end{pmatrix},
\begin{pmatrix}
0 & Y' & 0 & \ldots & 0 & 0
\end{pmatrix},\\
& \begin{pmatrix}
0 & 0 & Y' & \ldots & 0 & 0
\end{pmatrix},
\ldots
\begin{pmatrix}
0 & 0 & 0 & \ldots & Y' & 0
\end{pmatrix}, \ \mathrm{and}
& \begin{pmatrix}
0 & Y''\\
0 & 0
\end{pmatrix},
\end{align*}
where $Y'' \in \mathrm{M}_{r \times r}(F)$, one sees that $h_2$ and $h_1$ are of the form claimed.

We have shown earlier that $\dim(\mathrm{G}_{\overline{x}}) \geq \dim(\mathrm{V}_{\overline{x}})$.  Suppose that $\dim(\mathrm{G}_{\overline{x}}) > \dim(\mathrm{V}_{\overline{x}})$. Suppose that $\lambda \in \mathrm{V}_{\overline{x}}$ such that the orbit $\mathrm{G}_{\overline{x}} \cdot \lambda$ is Zariski-closed. Since $\dim(\mathrm{G}_{\overline{x}} \cdot \lambda) \leq \dim(\mathrm{V}_{\overline{x}}) - 1$, we deduce that $\mathrm{dim}(\mathrm{Stab}_{\mathrm{G}_{\overline{x}}}(\lambda)) \geq 2$, hence $\mathrm{Stab}_{\mathrm{G}_{\overline{x}}}(\lambda)$ is not finite modulo $\mathrm{ker}(N)$, since $\mathrm{ker}(N) = \RZ_x$.
Therefore, by Definition \ref{def3}, if $\dim(\mathrm{G}_{\overline{x}}) > \dim(\mathrm{V}_{\overline{x}})$, there is no stable functional in $\mathrm{V}_{\overline{x}}$.

We now assume that $\dim(\mathrm{G}_{\overline{x}}) = \dim(\mathrm{V}_{\overline{x}})$, so that $m_1 = m_2 = \cdots = m_k  =:m > 1$.
Since $\mathrm{dim}(\mathrm{ker}(N)) = 1$, we just have to show that there is no Zariski-closed orbit in $\mathrm{V}_{\overline{x}}$ of dimension $\dim(\mathrm{V}_{\overline{x}})-1$.

Assume that $\lambda = v(\overline{X}_1, \overline{X}_2, \ldots, \overline{X}_k) \in \mathrm{V}_{\overline{x}}$ such that $\mathrm{G}_{\overline{x}} \cdot \lambda$ is Zariski-closed. Since for any $g:=\diag(g_1, g_2, \ldots, g_k) \in \mathrm{G}_{\overline{x}}$, we have that
$$g \cdot v(\overline{X}_1, \overline{X}_2, \ldots, \overline{X}_k)=v(g_1\overline{X}_1g_2^{-1}, g_2\overline{X}_2g_3^{-1}, \ldots, g_k\overline{X}_kg_1^{-1}),$$
we see that the orbit $\mathrm{G}_{\overline{x}} \cdot \lambda$ is contained in the following set:
\begin{align*}
\begin{split}
S_1 := & \{v(W_1, W_2, \ldots, W_k) \in \mathrm{V}_{\overline{x}} :
W_1 W_{2} \cdots W_k = g_1 \overline{X}_1 \overline{X}_{2} \cdots \overline{X}_k g_1^{-1}, \\
& \text{ for some } g_1 \in \GL_{m_1}(\ol{k_F})\}.
\end{split}
\end{align*}
It is easy to see that the set $S_1$ is included in the following closed hyperplane of $\mathrm{V}_{\overline{x}}$ of dimension $\dim(\mathrm{V}_{\overline{x}})-1$:
\begin{align*}
\begin{split}
S_2 := & \{v(W_1, W_2, \ldots, W_k) \in \mathrm{V}_{\overline{x}} \\
& | \det(W_1 W_{2} \cdots W_k) = \det(\overline{X}_1 \overline{X}_{2} \cdots \overline{X}_k)\}.
\end{split}
\end{align*}
Therefore, it suffices to show that $S_1 \neq S_2$.

Note that $\overline{X}_1 \overline{X}_{2} \cdots \overline{X}_k \in \mathrm{M}_{m \times m}(\ol{k_F})$.
For any $v(W_1, W_2, \ldots, W_k) \in S_1$,
$W_1 W_{2} \cdots W_k$ has the same Jordan normal form as
$\overline{X}_1 \overline{X}_{2} \cdots \overline{X}_k$. Since $m > 1$, it is easy to find a $v(W_1, W_2, \ldots, W_k)$ such that $W_1 W_{2} \cdots W_k$ has a different Jordan normal form than $\overline{X}_1 \overline{X}_{2} \cdots \overline{X}_k$ satisfying $$\det(W_1 W_{2} \cdots W_k) = \det(\overline{X}_1 \overline{X}_{2} \cdots \overline{X}_k).$$
This shows that $S_1 \neq S_2$, and hence $\dim(\mathrm{G}_{\overline{x}} \cdot \lambda) \leq \dim(\mathrm{V}_{\overline{x}})-2$, which proves the claim.

Therefore, if $\lambda = v(\overline{X}_1, \overline{X}_2, \ldots, \overline{X}_k) \in \mathrm{V}_{\overline{x}}$ such that $\mathrm{G}_{\overline{x}} \cdot \lambda$ is Zariski-closed, we can conclude that $\mathrm{dim}(\mathrm{Stab}_{\mathrm{G}_{\overline{x}}}(\lambda)) \geq 2$, hence $\mathrm{Stab}_{\mathrm{G}_{\overline{x}}}(\lambda)$ is not finite modulo $\mathrm{ker}(N)$.  This completes the proof of the theorem.
\end{proof}

\subsection{Semi-stable functionals and nonbarycenters}\label{nonbarycenters}

In this section, we consider the situation of a point in the building that is not a barycenter.  The following theorem implies, by Remark \ref{moyprasadsemistable} and Theorem \ref{main2}, that the epipelagic supercuspidal representations of $\GL_n(F)$ obtained in Theorem \ref{thmRY13} can only be constructed from barycenters of alcoves.

\begin{thm}\label{main3}
If $x$ is not a barycenter, then $\check{\RV}_{x}$ contains no semi-stable functionals.
\end{thm}

\begin{proof}
Let $x$ be a nonbarycenter point contained in a facet $\mathcal{F}$.  We will show that $\mathrm{V}_{x} \subset \mathrm{V}_{\overline{x}}$, where $\ol{x}$ is a barycenter of $\CF$.  Assume that $\mathcal{F}$ is determined by the zero loci of $\Psi_{\mathcal{F}}$ as in \eqref{PsiF}.  Here, we are allowing $\mathcal{F}$ to be an alcove, so $\Psi_{\mathcal{F}}$ can be empty.  Recall that $\lambda \in \check{\RV}_{x}$ is \emph{semi-stable} if $\overline{\mathrm{G}_x \cdot \lambda}$ does not contain zero.  By the proof of Theorem \ref{main2}, it is enough to consider the case $t = 0$.
As in the proof of Theorem \ref{main2}, we identify $\mathrm{V}_{\overline{x}}$ with
$\check{\mathrm{V}}_{\overline{x}}$ via the standard trace pairing.

As in the proof of Theorem \ref{main2}, $x$ has the form $$(x_1, x_2, \ldots, x_k),$$
where $x_i = (s_i, s_i, \ldots, s_i) \in \mathbb{R}^{m_i}$.  Note that we have
$\sum_{i=1}^n \psi_i(x)=1$.
Also note that since $x$ is not a barycenter of $\CF$, $r(x) < \frac{1}{k}$.

Recall that a barycenter $\ol{x}$ of $\CF$ has the form as in \eqref{barycenter1}.
It is easy to see that $\RG_x = \RG_{\overline{x}}$.
We claim that $\RV_x \subseteq \RV_{\ol{x}}$ (see \eqref{Vx1}).

Since we have fixed a set of simple affine roots $\psi_1, \psi_2, \ldots, \psi_n$, we have in turn fixed an alcove $C$.  In particular,
$$C = \{ y \in \mathcal{A}(\mathrm{T},F) : 0 < \psi_i(y) < 1 \ \forall i = 1,2, \ldots, n \}.$$
Since $x \in \mathcal{F} \subset \overline{C}$, $s_1 \geq s_2 \geq \dots \geq s_k$.  By definition of $r(x)$, we have that $s_i - s_{i+1} \geq r(x)$ for $i = 1,2,...,k-1$.  Since $0 < \psi_n(x) < 1$, we get that $0 < s_1 - s_k < 1$.  Since $e_1 - e_n$ is the highest root, we conclude that the only positive roots that can contribute to $\mathrm{V}_{x}$ already contribute to $\mathrm{V}_{\ol{x}}$.  We now consider negative roots.  Since $e_1 - e_n$ is the highest root, the minimum value that a negative root can take on $x$ is $s_k - s_1$.  The negative roots that take the value $s_k - s_1$ on $x$ are precisely those that come from $\overline{X}_k$ in \eqref{Vx1}.  Since $\psi_n(x) \geq r(x)$, we have $e_n-e_1 \geq -1 + r(x)$.  In particular, all negative roots coming from $\overline{X}_k$ in \eqref{Vx1} have value greater than or equal to $-1 + r(x)$ on $x$, with equality if and only if $\psi_n(x) = r(x)$.
In particular, we have concluded that $\mathrm{V}_x \subseteq \mathrm{V}_{\overline{x}}$.  One can moreover see from the above discussion that since $x$ is not a barycenter, $\mathrm{V}_x \varsubsetneq \mathrm{V}_{\overline{x}}$.

We now consider the issue of semi-stability.  If $x$ is not a barycenter, then
\begin{equation}\label{Gx1}
\mathrm{G}_x = \mathrm{G}_{\overline{x}} \cong \prod_{i=1}^k \mathrm{GL}_{m_i}(\ol{k_F}) \cong \{\diag(g_1, g_2, \ldots, g_k) : g_i \in \mathrm{GL}_{m_i}, 1 \leq i \leq k\},
\end{equation}
and one can see that there exists $j$, with $1 \leq j \leq k$, such that
\begin{align}\label{Vx3}
\begin{split}
\mathrm{V}_{x} \subseteq \ & \{v(X_1, X_2, \ldots, X_k) : X_i \in \mathrm{M}_{m_i \times m_{i+1}}(\ol{k_F}), \\
& \ 1 \leq i \leq k-1, X_k \in \mathrm{M}_{m_k \times m_1}(\ol{k_F}), X_j = 0 \},
\end{split}
\end{align}
with
$$g \cdot v(X_1, X_2, \ldots, X_k)=v(g_1X_1g_2^{-1}, g_2X_2g_3^{-1}, \ldots, g_kX_kg_1^{-1})$$
for any $g \in \mathrm{G}_{x}$ and $v \in \mathrm{V}_{x}$ as above.

We wish to show that $\overline{\mathrm{G}_x \cdot \lambda}$ contains zero.  Let $\RT_x$ be the split maximal torus of $\mathrm{G}_x$.  Let $\chi : \mathbb{G}_m \rightarrow \RT_x$ be the cocharacter given by $\chi(t) = (t_1 I_{m_1}, t_2 I_{m_2}, \ldots, t_k I_k)$, where $t_i \in \ol{k_F}^{\times}$, $t_i = t^{b_i}$, and $b_i \in \mathbb{Z}$.  Then
\begin{align*}
\ & \chi(t) \cdot v(X_1, X_2, \ldots, X_{k-1}, X_k) \\
= \ & v(t^{b_1 - b_2} X_1, t^{b_2 - b_3} X_2, \ldots, t^{b_{k-1} - b_k} X_{k-1}, t^{b_k-b_1} X_k).
\end{align*}
If $1 \leq j \leq k-1$, let  $(b_1, b_2, \ldots, b_k)$ be any sequence of integers that satisfy $b_{j+1} \geq b_{j+2} \geq \ldots \geq b_k \geq b_1 \geq b_2 \geq \ldots \geq b_{j}$.  If $j = k$, let $(b_1, b_2, \ldots, b_k)$ be any sequence of integers that satisfy $b_1 \geq b_2 \geq \ldots b_{k-1} \geq b_k$.  Then, since $X_j = 0$, one can compute that
\[ \lim_{t \to 0} [\chi(t) \cdot v(X_1, X_2, \ldots, X_{k-1}, X_k)] = 0,\]
proving that $\overline{\mathrm{G}_x \cdot \lambda}$ contains zero.  Thus, $\lambda$ is not semistable.
\end{proof}

\begin{rmk}\label{epipelagicisnotsimple} $ $
\begin{enumerate}
\item By comparing the construction of epipelagic representations with that of simple supercuspidal representation in \S\ref{prelimsimplesupercuspidal}, one can see that a barycenter of an alcove $C$ is exactly a point which produces simple supercuspidal representations.  Therefore, by Theorem \ref{main2} and Theorem \ref{main3}, we conclude that the irreducible epipelagic supercuspidal representations of $\GL_n(F)$ constructed in \S\ref{constructionepipelagic} are all simple.
\item The epipelagic supercuspidal representations constructed in Theorem \ref{thmRY13} do not exhaust all epipelagic supercuspidal representations of $\GL_n(F)$.  For example, in $\GL_4(F)$, one could take $x$ to be a point corresponding to the barycenter of a facet whose reductive quotient is $\GL_2 \times \GL_2$. Then $r(x)=1/2$ and there are supercuspidals of $\GL_4(F)$ of depth $1/2$ constructed from $x$, as pointed out to us by Shaun Stevens.
\item By considering the case of $\SL_n(F)$, we can now conclude that the construction of Reeder and Yu is not necessarily exhaustive.
\end{enumerate}
\end{rmk}

\section{Jacquet's Conjecture on the local converse problem for simple supercuspidal representations of $\GL_n(F)$}\label{jacquetsection}

In this section, we prove a new case of Jacquet's conjecture on the local converse problem for $\GL_n(F)$.  In \S\ref{jacquetpreliminaries}, we recall some basic theory on Jacquet's conjecture, as well as a new strategy developed by Jiang, Nien, and Stevens, on proving the conjecture.  In \S\ref{specialpairsimple}, we prove Jacquet's conjecture in the case of simple supercuspidal representations of $\GL_n(F)$, using the strategy of Jiang, Nien, and Stevens.

\subsection{Prelminaries on Jacquet's conjecture}\label{jacquetpreliminaries}
Let $F$ be a nonarchimedean local field of characteristic zero, and fix an additive character $\psi$ of $F$.  Let $G_n:=\GL_n(F)$ and let $\pi$ be an irreducible admissible generic representation of $G_n$.
For any irreducible admissible generic representation $\tau$ of $G_r$, a family of local gamma factors $\gamma(s, \pi \times \tau, \psi)$ can be defined using Rankin-Selberg convolution
\cite{JPSS83} or the Langlands-Shahidi method \cite{S84}.
Jacquet has formulated the following conjecture on precisely which family of local gamma factors should uniquely determine $\pi$ (see \cite{JNS13} for more related discussion).

\begin{conj}[The Jacquet Conjecture on the Local Converse Problem]\label{lcp1}
Let $\pi_1$ and $\pi_2$ be irreducible admissible generic representations of $G_n$. If
$$\gamma(s, \pi_1 \times \tau, \psi) = \gamma(s, \pi_2 \times \tau, \psi),$$
for any irreducible admissible generic representation $\tau$ of $G_r$ with $r = 1, \ldots, [\frac{n}{2}]$,
then $\pi_1 \cong \pi_2$.
\end{conj}

In \cite{JNS13}, Conjecture \ref{lcp1} is shown to be equivalent to the following conjecture.

\begin{conj}\label{lcp2}
Let $\pi_1$ and $\pi_2$ be irreducible unitarizable supercuspidal representations of $G_n$. If
$$\gamma(s, \pi_1 \times \tau, \psi) = \gamma(s, \pi_2 \times \tau, \psi),$$
for any irreducible supercuspdial representation $\tau$ of $G_r$ with $r = 1, \ldots, [\frac{n}{2}]$,
then $\pi_1 \cong \pi_2$.
\end{conj}

Jiang, Nien and Stevens \cite{JNS13} have formulated a general approach to prove Conjecture \ref{lcp2}. With this approach, Conjecture \ref{lcp2} is proven under an assumption which has been verified in several cases, including the case of depth zero supercuspidal representations. In this section, we verify this assumption for the case of simple supercuspidal representations.
To make things precise, we first need to recall some notation as follows.

Let $B_n=T_nU_n$ be the Borel subgroup of $G_n$
which consists of upper triangular matrices, where $T_n$ consists of all diagonal matrices and $U_n$ is the unipotent radical of $B_n$. For any $u \in U_n$, let
$$\psi_{U_n}(u)=\psi(\sum_{i=1}^{n-1} u_{i,i+1}).$$

An irreducible admissible representation $(\pi, V_{\pi})$ of $G_n$ is called \emph{generic} if $\Hom_{G_n}(V_{\pi}, \Ind_{U_n}^{G_n} \psi_{U_n}) \neq 0.$
By the uniqueness of local Whittaker models, this Hom-space is at most one dimensional. By Frobenius reciprocity,
$$\Hom_{G_n}(V_{\pi}, \Ind_{U_n}^{G_n} \psi_{U_n})\cong \Hom_{U_n} (V_{\pi}|_{U_n}, \psi_{U_n}).$$
Therefore, $\Hom_{U_n} (V_{\pi}|_{U_n}, \psi_{U_n})$ is also at most one dimensional.

Assume that $(\pi, V_{\pi})$ is generic. Fix a nonzero functional
$$l \in \Hom_{U_n} (V_{\pi}|_{U_n}, \psi_{U_n}),$$
which is unique up to scalar.
The Whittaker function attached to a vector $v \in V_{\pi}$ is defined by
$$W_v(g):= l(\pi(g)v), \text{ for all } g \in G_n.$$
It is easy to see that $W_v \in \Ind_{U_n}^{G_n} \psi_{U_n}$.
The space
$$\mathcal{W}(\pi, \psi_{U_n}) :=\{W_v| v \in V_{\pi}\}$$
is called the \emph{Whittaker model} of $\pi$, and $G_n$ acts on it by right translation.
It is easy to see that the Whittaker model of $\pi$ is independent of the choice of the nonzero functional $l$.

Jiang, Nien and Stevens \cite{JNS13} introduced the notion of a $K$-special Whittaker function as follows.

\begin{defn}\label{def1}
Let $\pi$ be an irreducible unitarizable supercuspidal representation of $G_n$ and let $K$ be a compact-mod-center open subgroup of $G_n$. A nonzero Whittaker function $W_{\pi}$ for $\pi$ is called \emph{$K$-special} if the support of $W_{\pi}$ satisfies $\mathrm{Supp} (W_{\pi}) \subset U_n K$, and if
$$W_{\pi}(k^{-1})=\overline{W_{\pi}(k)} \text{ for all } k \in K,$$
where $\overline{z}$ denotes the complex conjugate of $z \in \BC$.
\end{defn}

Let $P_n$ be the mirabolic subgroup of $G_n$ which consists of matrices with last row equal to $(0, \ldots, 0, 1)$.

\begin{defn}\label{def2}
Let $(\pi_1, \pi_2)$ be a pair of irreducible unitarizable supercuspidal representations of $G_n$ with the same central character. Let $W_{\pi_1}$ and $W_{\pi_2}$ be nonzero Whittaker functions for $\pi_1$ and $\pi_2$, respectively.
$(W_{\pi_1}, W_{\pi_2})$ is called a \emph{special pair} of Whittaker functions for $(\pi_1, \pi_2)$ if there exists a compact-mod-center open subgroup $K$ of $G_n$ such that $W_{\pi_1}$ and  $W_{\pi_2}$ are both $K$-special and
$$W_{\pi_1}(p)=W_{\pi_2}(p), \text{ for all } p \in P_n.$$
\end{defn}

The following theorem is one of the main results of Jiang, Nien and Stevens in \cite{JNS13}, which provides a general approach to prove Conjecture \ref{lcp2}.

\begin{thm}[Jiang, Nien, and Stevens \cite{JNS13}]\label{thmJNS13}
Let $(\pi_1, \pi_2)$ be a pair of irreducible unitarizable supercuspidal representations of $G_n$ with the same central character. Assume that there exists a special pair of Whittaker functions $(W_{\pi_1}, W_{\pi_2})$ for $(\pi_1, \pi_2)$. If
$$\gamma(s, \pi_1 \times \tau, \psi) = \gamma(s, \pi_2 \times \tau, \psi),$$
for any irreducible supercuspdial representation $\tau$ of $G_r$ with $r = 1, \ldots, [\frac{n}{2}]$,
then $\pi_1 \cong \pi_2$.
\end{thm}

Using Theorem \ref{thmJNS13}, Jiang, Nien, and Stevens are able to prove that if $deg(\pi_1) < n$ (see \cite[Section 4.2]{JNS13} for the definition of degree of an irreducible supercuspidal representation of $\GL_n(F)$), then Conjecture \ref{lcp2} is true.  We will show that any pair $(\pi_1, \pi_2)$ of irreducible unitarizable simple supercuspidal representations of $G_n$ with the same central character admits a special pair of Whittaker functions.  Since the degree of a simple supercuspidal representation is $n$, our result implies that Conjecture \ref{lcp2} is true for a case that is not covered by \cite{JNS13}.  Our main theorem is the following.

\begin{thm}\label{thmepi}
Let $(\pi_1, \pi_2)$ be a pair of irreducible unitarizable simple supercuspidal representations of $G_n$ with the same central character. Then there exists a special pair of
Whittaker functions $(W_{\pi_1}, W_{\pi_2})$ for $(\pi_1, \pi_2)$.
\end{thm}

\begin{rmk} $ $
\begin{enumerate}
\item Although Theorem \ref{thmepi} implies that Conjecture \ref{lcp2} is true in a new setting, our results from \S\ref{llcsimple} are actually much stronger.  Indeed, because of Remark \ref{simplesupercuspidaldetermined}, \cite[Corollary 2.7]{JNS13}, \cite[Theorem 3.1]{LR03}, and the fact that the simple supercuspidal representations of $\GL_n(F)$ exhaust all depth $\frac{1}{n}$ supercuspidal representations of $\GL_n(F)$, we can conclude the following result pertaining to the local converse problem for $\GL_n(F)$.  Suppose that $\pi_1$ is simple supercuspidal and $\pi_2$ is supercuspidal. If $$\gamma(s, \pi_1 \times \lambda, \psi) = \gamma(s, \pi_2 \times \lambda, \psi)$$ for all characters $\lambda$ of $F^{\times}$, then $\pi_1 \cong \pi_2$.  A version of this result can also be found in \cite[Proposition 2.2]{BH13}.  Peng Xu has also independently obtained the same result, for $p \nmid n$, in \cite{X13}.
\item Even though the application of Theorem \ref{thmepi} to Jacquet's Conjecture is subsumed by our work in \S\ref{llcsimple}, we find the statement of Theorem \ref{thmepi} to be interesting in and of its own right.  It does not seem clear to us that if two supercuspidal representations of $\GL_n(F)$ can be induced from different compact open subgroups, then they could have a special pair of Whittaker functions.
\item To prove Conjecture \ref{lcp2}, it suffices to consider supercuspidal representations $\pi_1, \pi_2$ of the same depth.  The reason is that if $\gamma(s, \pi_1, \psi) = \gamma(s, \pi_2, \psi)$ as functions in $s$, then by \cite[Theorem 3.1]{LR03}, $\pi_1$ must have the same depth as $\pi_2$.
\end{enumerate}
\end{rmk}

\subsection{The case of simple supercuspidal representations}\label{specialpairsimple}
Recall that for a fixed central character, the simple supercuspidal representations were parameterized by a choice of uniformizer $\varpi$ and a complex root of unity $\zeta$.  Since we will be dealing with multiple simple supercuspidal representations at the same time in this section, we will denote by $\pi_{\varpi, \zeta}$ (see \S\ref{prelimsimplesupercuspidal}) the corresponding simple supercuspidal.  Recall also that $\varpi$ determines an affine generic character $\chi$, from which we defined an element $g_{\chi}$.  In this section, we will write $g_{\varpi}$ instead of $g_{\chi}$.  Given $\zeta$, we also defined $\chi_{\zeta}$, which we will now call $\chi_{\varpi, \zeta}$.  For the irreducible simple supercuspidal representation $\pi_{\varpi, \zeta}$, a Whittaker function has been defined in \S\ref{computingepsilonfactors} as follows:
\begin{equation*}
W_{\varpi, \zeta}(g) = \left\{
\begin{array}{rll}
\psi_{U}(u) \chi_{\varpi, \zeta}(h') & \text{if} & g = uh' \in U H'\\
0 &  & \text{else}
\end{array} \right.
\end{equation*}
where $H' = \langle g_{\varpi} \rangle Z I^+$.  The following theorem is the main result of this section.

\begin{thm}\label{main1}
Let $\varpi_i$, for $i = 1,2$ be two uniformizers of $F$, and let $\zeta_i$ be an $n^{\mathrm{th}}$ root of $\omega(\varpi_i)$, for $i = 1,2$.  For any pair of irreducible unitarizable simple supercuspidal representations $(\pi_{\varpi_1, \zeta_1}, \pi_{\varpi_2, \zeta_2})$ of $\GL_n(F)$ with the same given unitary central character $\omega$, $(W_{\varpi_1, \zeta_1},W_{\varpi_2, \zeta_2})$ is a special pair of Whittaker functions.
\end{thm}

\begin{proof}
Let $K = \langle g_{\varpi_1}, g_{\varpi_2} \rangle Z I^+$ and $K_i = \langle g_{\varpi_i} \rangle Z I^+$, for $i = 1,2$.  It is not difficult to see that $K$ is a compact-mod-center open subgroup of $\GL_n(F)$.  Noting that $\mathrm{Supp} W_{\varpi_i, \zeta_i} \subset U_n K_i \subset U_nK$, it suffices
by Definitions \ref{def1} and \ref{def2} to prove the following:

(1) $W_{\varpi_i, \zeta_i}(k^{-1})=\ol{W_{\varpi_i, \zeta_i}(k)}$, for any $k \in K$, for $i=1,2$;

(2) $W_{\varpi_1, \zeta_1}(p)=W_{\varpi_2, \zeta_2}(p)$, for any $p \in P_n$.

For (1), note that by definition, for $i=1,2$,
\begin{equation*}
W_{\varpi_i, \zeta_i}(k) = \left\{
\begin{array}{rll}
\chi_{\varpi_i,\zeta_i}(k) & \text{if} & k \in K_i\\
0 & \text{if} & k \in K - K_i
\end{array} \right.
\end{equation*}
Therefore, we just need to show that $\chi_{\varpi_i,\zeta_i}$ is unitary.

For any $k=g_{\varpi_i}^j z g \in K_i$, where $z \in Z$, $g = (g_{k,\ell}) \in I^+$,
$$\chi_{\varpi_i,\zeta_i}(k)=\zeta_i^j \omega(z)
\psi(\sum_{\ell=1}^{n-1} g_{\ell,\ell+1} + g_i),$$
where $g_{n,1}=\varpi_i \cdot g_i$.
Since $\omega$ is unitary, and since $\zeta_i$ is an $n$-th root of $\omega(\varpi_i)$, we get that $\zeta_i \in \BC^1$, where $\BC^1$ is the set of all elements in $\BC$ with complex norm $1$.
Since $\psi$ is level one, it is easy to see that $\chi_{\varpi_i,\zeta_i}|_{I^+}$ factors to a character of the finite group $k_F^n$, hence is automatically unitary.
Therefore, $\chi_{\varpi_i,\zeta_i}$ is a unitary character.

For (2), note that $P_n = \iota(G_{n-1}) N_{n-1,1}$,
where
$$\iota(G_{n-1}) = \{\iota(g) := \begin{pmatrix}
g & 0\\
0 & 1
\end{pmatrix}| g \in G_{n-1}\},$$
and where $N_{n-1,1}$ is the unipotent radical of the parabolic subgroup $Q_{n-1,1}$ with Levi subgroup $G_{n-1} \times G_1$.

It is clear that for $x \in N_{n-1,1}$,
$$W_{\varpi_1, \zeta_1}(x)=W_{\varpi_2, \zeta_2}(x)=\psi(x_{n-1,n}).$$
Therefore, since $\iota(G_{n-1})$ normalizes $N_{n-1,1}$, and since $W_{\varpi_i, \zeta_i}$ transform on the left by $\psi_{U_n}$ on $U_n$, it remains to show that
$$W_{\varpi_1, \zeta_1}(\iota(g))=W_{\varpi_2, \zeta_2}(\iota(g)),$$
for any $g \in G_{n-1}$.

Note that $K_i = \langle g_{\varpi_i} \rangle Z I^+$, and we have a partition $\langle g_{\varpi_i} \rangle Z = \displaystyle\coprod_{j=1}^n g_{\varpi_i}^j Z$. And for $j=1,2,\ldots, n$,
$$g_{\varpi_i}^j = \begin{pmatrix}
0 &  I_{n-j}\\
\varpi_i I_j
\end{pmatrix}.$$
By Lemma \ref{affinebruhat} applied to $G_{n-1}$, we can see that
$$\iota(G_{n-1}) \cap U_n K_i
=\iota(G_{n-1}) \cap U_n \langle g_{\varpi_i} \rangle Z I^+
\subset \iota(U_{n-1} I_{(n-1)}^+),$$ where $I_{(n-1)}^+$ denotes the pro-unipotent part of the standard Iwahori subgroup of $\GL_{n-1}(F)$.

By definition, for $i=1,2$,
for any $u \in U_{n-1}$, $g \in I_{(n-1)}^+$,
$$W_{\varpi_i, \zeta_i}(\iota(ug)) = \psi(\sum_{j=1}^{n-2} u_{j,j+1})\psi(\sum_{s=1}^{n-2} g_{s,s+1}),$$
which is independent of $\varpi_i$ and $\zeta_i$.
Therefore, (2) is also proved.

This completes the proof of the theorem.
\end{proof}


\begin{thebibliography}{}
\bibitem[A13] {A13}
M. Adrian,
{\it On the local constants of simple supercuspidal representations of $\GL_n(F)$.}
Preprint. 2013.
\bibitem[B08] {B08}
D. Bump,
{\it Automorphic Forms and Representations.}
Cambridge Studies in Advanced Mathematics, Volume \textbf{55}, Cambridge University Press, 2008.
\bibitem[BF83] {BF83}
C. Bushnell and A. Fr\"{o}hlich,
{\it Gauss sums and p-adic division algebras.}
Lecture Notes in Mathematics, \textbf{987}. Springer-Verlag, Berlin-New York, 1983. xi+187 pp. ISBN: 3-540-12290-7.
\bibitem[BH98] {BH98}
C. Bushnell and G. Henniart,
{\it Supercuspidal representations of $\GL_n$: explicit Whittaker functions.}
J. Algebra \textbf{209} (1998), 270--287.
\bibitem[BH05i] {BH05i}
C. Bushnell and G. Henniart,
{\it The essentially tame local Langlands correspondence, I.}
J. Amer. Math. Soc. \textbf{18} (2005), no. 3, 685--710.
\bibitem[BH05ii] {BH05ii}
C. Bushnell and G. Henniart,
{\it The essentially tame local Langlands correspondence, II: totally ramified representations.}
Compositio Math. \textbf{141} (2005) 979--1011.
\bibitem[BH06] {BH06}
C. Bushnell and G. Henniart,
{\it The Local Langlands Conjecture for $\GL(2)$.}
A Series of Comprehensive Studies in Mathematics, Volume \textbf{335}, Springer Berlin Heidelberg, 2006.
\bibitem[BH10] {BH10}
C. Bushnell and G. Henniart,
{\it The essentially tame local Langlands correspondence, III: the general case.}
Proc. Lond. Math. Soc. (3) \textbf{101} (2010), no. 2, 497--553.
\bibitem[BH13] {BH13}
C. Bushnell and G. Henniart,
{\it Langlands parameters for epipelagic representations of $\GL_n$.}
Preprint. 2013.
\bibitem[C79] {C79}
H. Carayol,
{\it Representations supercuspidales de $\GL_n$.}
C.R. Acad. Sc. Paris, serie A, t. \textbf{288}, 1979, 17--20.
\bibitem[Cog00] {Cog00}
J. Cogdell,
{\it Notes on $L$-functions for $\mathrm{GL}_n$.}
Lectures given at the School of Automorphic Forms on $\GL(n)$, Trieste, 2000.
\bibitem[GR10] {GR10}
B. Gross and M. Reeder,
{\it Arithmetic invariants of discrete Langlands parameters.}
Duke Math. J. \textbf{154} (2010), no. 3, 431--508.
\bibitem[HT01] {HT01}
M. Harris and R. Taylor,
{\it On the geometry and cohomology of some simple Shimura varieties.}
Ann. of Math. Studies \textbf{151}, Princeton Univ. Press, Princeton, NJ (2001).
\bibitem[H00] {H00}
G. Henniart,
{\it Une preuve simple des conjectures de Langlands for $\GL(n)$ sur un corps $p$-adique.}
Invent. Math. \textbf{139} (2000), 439--455.
\bibitem[JPSS83] {JPSS83}
H. Jacquet, I. Piatetski-Shapiro and J. Shalika,
{\it Rankin--Selberg convolutions.}
Amer. J. Math. \textbf{105} (1983), 367--464.
\bibitem[JNS13] {JNS13}
D. Jiang, C. Nien and S. Stevens,
{\it Towards the Jacquet Conjecture on the Local Converse Problem for $p$-Adic $\GL_n$.}
To appear in Journal of the European Mathematical Society. 2013.
\bibitem[K13] {K13}
T. Kaletha,
{\it Epipelagic $L$-packets and rectifying characters.}
Preprint. 2013.
\bibitem[KL13] {KL13}
A. Knightly and C. Li,
{\it Simple supercuspidal representations of $\GL(n)$.}
Preprint. 2013.
\bibitem[LR03] {LR03}
J. Lansky and A. Raghuram,
{\it A remark on the correspondence of representations between
$\GL(n)$ and division algebras.}
Proc. Am. Math. Soc. \textbf{131}(5) (2003) 1641–1648.
\bibitem[M86] {M86}
A. Moy,
{\it Local Constants and the Tame Langlands Correspondence.}
American Journal of Math.  \textbf{108}  (1986),  no. 4, 863--929.
\bibitem[MP94] {MP94}
A. Moy and G. Prasad,
{\it Unrefined minimal $K$-types for $p$-adic groups.}
Inv. Math., \textbf{116} (1994), 393--408.
\bibitem[PS08] {PS08}
V. Paskunas and S. Stevens,
{\it On the realization of maximal simple types and epsilon factors of pairs.}
Amer. J. Math. \textbf{130}(5) (2008) 1211-1261.
\bibitem[RY13] {RY13}
M. Reeder and J.-K Yu,
{\it Epipelagic representations and invariant theory.}
J. Amer. Math. Soc., to appear.
\bibitem[S84] {S84}
F. Shahidi,
{\it Fourier transforms of intertwining operators and Plancherel measures for $\GL(n)$.}
Amer. J. Math. \textbf{106} (1984), no. 1, 67--111.
\bibitem[X13] {X13}
P. Xu,
{\it A remark on the simple cuspidal representations of $GL_n$.}
Preprint. 2013.
\end{thebibliography}
\end{document}